\documentclass{amsart}
\usepackage{graphicx}
\usepackage{bbm}
\usepackage{epsfig}
\usepackage{amssymb}
\usepackage{amsmath,enumerate,color}
\usepackage{epstopdf}
\usepackage{amsfonts}
\usepackage{siunitx}
\usepackage{booktabs}
\usepackage{mathrsfs}
\usepackage{cases}
\usepackage{subcaption}
\usepackage{caption}
\usepackage{graphicx}
\usepackage{hyperref}
\usepackage{comment,enumerate,multicol,xspace}
\usepackage{soul}
\usepackage{enumitem}
\usepackage[ruled,linesnumbered]{algorithm2e}

\newcounter{mnote}
\let\oldmarginpar\marginpar
\renewcommand\marginpar[1]{\-\oldmarginpar[\raggedleft\footnotesize #1]%
{\raggedright\footnotesize #1}}

\newtheorem{theorem}{Theorem}[section]
\newtheorem{lemma}[theorem]{Lemma}

\theoremstyle{definition}

\theoremstyle{remark}
\newtheorem{remark}[theorem]{Remark}

\numberwithin{equation}{section}

\SetKwInput{KwIn}{Input}
\SetKwInput{KwOut}{Output}
\SetKwRepeat{Do}{do}{while}
\SetAlgoCaptionSeparator{: } 

\begin{document}

\title[Forward-Backward Stochastic Jump Neural Networks]{FBSJNN: A Theoretically Interpretable and Efficiently Deep Learning method for Solving Partial Integro-Differential Equations}
\thanks{This work was supported by the National Natural Science Foundation of China (Grant Nos. 12271367, 92470119)}

\author{Zaijun Ye}
\address{Department of Mathematics, Shanghai Normal University, Shanghai, 200234, China}
\email{yezaijun@outlook.com}

\author{Wansheng Wang}
\address{Department of Mathematics, Shanghai Normal University, Shanghai, 200234, China}
\email{w.s.wang@163.com}
\thanks{The second author is the corresponding author.}


\begin{abstract}
    We propose a novel framework for solving a class of Partial Integro-Differential Equations (PIDEs) and Forward-Backward Stochastic Differential Equations with Jumps (FBSDEJs) through a deep learning-based approach. This method, termed the Forward-Backward Stochastic Jump Neural Network (FBSJNN), is both theoretically interpretable and numerically effective. Theoretical analysis establishes the convergence of the numerical scheme and provides error estimates grounded in the universal approximation properties of neural networks. In comparison to existing methods, the key innovation of the FBSJNN framework is that it uses a single neural network to approximate both the solution of the PIDEs and the non-local integral, leveraging Taylor expansion for the latter. This enables the method to reduce the total number of parameters in FBSJNN, which enhances optimization efficiency. Numerical experiments indicate that the FBSJNN scheme can obtain numerical solutions with a relative error on the scale of $10^{-3}$.
\end{abstract}

\maketitle

\section{Introduction}

In this paper, we investigate the deep learning method for solving Partial Integro-Differential Equations (PIDEs).
PIDEs are widely applicable to model phenomena that exhibit both spatial variation and memory effects, making them valuable tools across fields such as finance, engineering, biology, and so on.
However, due to the ``curse of dimensionality", obtaining numerical solutions for high-dimensional PIDEs remains challenging and requires urgent attention.
By employing the well-known nonlinear Feynman-Kac formula (see, e.g., \cite{pengProbabilisticInterpretationSystems1991, barlesBackwardStochasticDifferential1997}), the PIDEs can be transformed into Forward-Backward Stochastic Differential Equations with Jumps (FBSDEJs), and vice versa, providing multiple perspectives to address these challenges more effectively.

According to the Universal Approximation Theorem (UAT, see, e.g., \cite{cybenkoApproximationSuperpositionsSigmoidal1989, hornikMultilayerFeedforwardNetworks1989, stinchcombeUniversalApproximationUsing1989}), Neural Networks (NNs) with sufficient width are dense in the space of continuous functions, which makes NNs a valuable tool for approximating solutions to the PIDEs and FBSDEJs (see, e.g., \cite{jentzenProofThatDeep2019,kutyniokTheoreticalAnalysisDeep2022,shinChapter6Theoretical2024}).
Several deep learning methods, inspired by deep neural networks, have been successfully applied to solve PIDEs and FBSDEJs, including: 
Deep Ritz Method (see, e.g., \cite{eDeepRitzMethod2018, duanConvergenceRateAnalysis2022}),
Deep Galerkin Method (see, e.g., \cite{sirignanoDGMDeepLearning2018, saporitoPathDependentDeepGalerkin2021}),
Deep Splitting Methods (see, e.g., \cite{beckDeepSplittingMethod2021, freyConvergenceAnalysisDeep2022}),
Physics-Informed Neural Networks (PINN, see, e.g., \cite{raissiPhysicsinformedNeuralNetworks2019, luDeepXDEDeepLearning2021}),
Deep BSDE Method (Deep BSDE, see, e.g., \cite{eDeepLearningBasedNumerical2017, hanSolvingHighdimensionalPartial2018,hanConvergenceDeepBSDE2020,anderssonConvergenceRobustDeep2023}),
Deep Backward Schemes (DBDP, see, e.g., \cite{hureDeepBackwardSchemes2020, phamNeuralNetworksbasedBackward2021, germainApproximationErrorAnalysis2022}),
and so on.
These approaches represent significant advancements in addressing a wide range of problems, including those in high-dimensional cases. 

While extensive works (see, e.g., \cite{
    chan-wai-namMachineLearningSemi2019,
    beckMachineLearningApproximation2019,
    castroDeepLearningSchemes2022, 
    freyDeepNeuralNetwork2022, 
    germainApproximationErrorAnalysis2022,
    luTemporalDifferenceLearning2023, 
    wangDeepLearningNumerical2023, 
    kapllaniDeepLearningAlgorithms2024,
    neufeldRectifiedDeepNeural2025}) have advanced the theoretical foundations and variations of these methods, they still exhibit certain limitations in various aspects.
Methods inspired by PINNs \cite{raissiPhysicsinformedNeuralNetworks2019} model solutions to PDEs as differentiable functions, embedding them into the problem and boundary conditions to define loss terms. Despite their intuitive formulation, these methods lack strong theoretical guarantees for solution accuracy and behavior.
Other methods, like those inspired by Deep BSDE \cite{eDeepLearningBasedNumerical2017} and DBDP \cite{hureDeepBackwardSchemes2020} approaches, have demonstrated good convergence and consistency, but they come with drawbacks in practice. 
The iterative structure of Deep BSDE approaches, which approximate gradient terms and propagate the solution to a terminal condition, creates a highly nested network structure that complicates optimization. 
Meanwhile, approaches that decompose FBSDEs into sequential steps are susceptible to cumulative error and often require manual tuning to achieve optimal results, as optimization methods do not always converge as expected.  

To address these limitations, we propose a novel framework that aims to balance theoretical properties of the numerical solution with computational efficiency in solving complex high-dimensional PIDEs and FBSDEJs. Referring to the work \cite{raissiForwardBackwardStochastic2018, broux-quemeraisDeepLearningScheme2024}, after transforming the PIDEs into FBSDEJs, we discretize time using the Euler-Maruyama scheme and directly employ neural networks to approximate the numerical solution. The loss is computed and optimized within each discretized interval. Automatic differentiation is utilized to calculate gradients, and Taylor expansion is applied to simplify non-local integral operations. We also conduct a detailed analysis to confirm the consistency of the proposed method. Further numerical experiments verify the effectiveness of the numerical approach developed in this paper. We called this method Forward-Backward Stochastic Jump Neural Networks (FBSJNN). Obviously, this method can be also applied to PDEs and FBSDEs without jump.

This work not only proposes a theoretically interpretable method based on deep learning but also provides new insights into the relationship between two neural network functions. The first is the neural network function obtained through optimizing a loss function, while the second is the neural network function guaranteed by the universal approximation theorem, which can approximate a target function with arbitrary accuracy. We find that the error of the optimized function can be bounded by the error of the universally approximating function. Our theoretical analysis shows that as the time step approaches zero, the errors of both functions converge, highlighting a previously unexplored connection. This discovery deepens the theoretical understanding of using deep learning to solve PDEs.

The organization of this paper is as follows. In Section \ref{section:Forward-Backward Stochastic Jump Neural Networks for Solving PIDEs}, we will introduce the core concepts of the proposed framework and establish the associated numerical scheme. Following this, Section \ref{section:Consistency Analysis} will provide a theoretical analysis of the numerical scheme, yielding error estimates and verifying the consistency of the proposed approach. Finally, a series of numerical examples will be presented in Section \ref{section:Numerical Examples} to demonstrate the effectiveness of the proposed algorithm.

\section{Forward-backward stochastic jump neural networks}
\label{section:Forward-Backward Stochastic Jump Neural Networks for Solving PIDEs}
In this section, we explore the forward-backward stochastic jump neural networks (FBSJNNs) to solve PIDEs and FBSDEJs. First, we utilize the nonlinear Feynman-Kac formula to transform PIDEs into FBSDEJs. Then we use the Euler-Maruyama scheme to obtain the discrete approximation of FBSDEJs. By approximating the solution, gradients, and non-local integrals in different ways, we use deep learning techniques to solve the discrete FBSDEJs, and ultimately obtain the numerical solutions of PIDEs.

\subsection{PIDEs} We consider the following second-order semi-linear parabolic PIDEs
\begin{equation}\label{eqn:PIDE}
    \left\{
    \begin{aligned}
        -\partial_tu-\mathcal{L}u-f\left(t,x,u,\sigma^\top\nabla_x u,I_t\right) & =0,    & (t,x) & \in[0,T]\times\mathbb{R}^d, \\
        u(T,x)                                                           & =g(x), & x     & \in\mathbb{R}^d,
    \end{aligned}
    \right.
\end{equation}
where, $d\geq1$ denotes the dimension, $T>0$ is the terminal time, $g:\mathbb{R}^d\mapsto\mathbb{R}$, and $f:[0,T]\times\mathbb{R}^d\times\mathbb{R}\times\mathbb{R}^d\times\mathbb{R}\mapsto\mathbb{R}$. Here $E := \mathbb{R}^l\setminus\{0\}$, and the integro-differential operator $\mathcal{L}$ is defined as follows:
$$
\begin{aligned}
    \mathcal{L}u := & \; \frac{1}{2}\mathrm{Tr}\left(\sigma\sigma^\top\nabla_x^2u\right)+\big\langle b,\nabla_xu \big\rangle          \\
                    & + \int_E \left(u\left(t,x+\beta(t,x,e)\right) -u(t,x) - \left\langle \nabla_x u, \beta(t,x,e)\right\rangle \right)\lambda(de),
\end{aligned}
$$
where $\mathrm{Tr}(\cdot)$ denotes the trace of a matrix, $\sigma:[0,T]\times\mathbb{R}^d\mapsto\mathbb{R}^{d\times d}$ is a matrix function, $\nabla_x^2u$ is the Hessian matrix of $u$, $\langle \cdot,\cdot \rangle$ denotes the inner product, $b:[0,T]\times\mathbb{R}^d\mapsto \mathbb{R}^d$ and $\beta:[0,T]\times\mathbb{R}^d\times E\mapsto \mathbb{R}^d$ are vector functions and $\lambda(de)$ is the $\sigma$-finite L\'{e}vy measure defined on $(E,\mathcal{E})$ with $\mathcal{E}$ being the Borel field corresponding to $E$. The integral operator $I_t$ is formally defined as follows:
$$I_t := \int_E \left[u\left(t,x+\beta(t,x,e)\right) -u(t,x) \right] \lambda(de) .$$

\subsection{FBSDEJs}
Let us consider a probability space $\left(\Omega,\mathcal{F},(\mathcal{F}_t)_{t\geq0},\mathbb{P}\right)$. The filtration $\mathbb{F} := (\mathcal{F}_t)_{t\geq0}$ is generated by two mutually independent stochastic processes: a standard $d$-dimensional Brownian motion $\{W_t\}_{t\geq 0}$ and a Poisson measure $\mu$ defined on $\mathbb{R}_+\times E$. We introduce the compensated random measure $\nu(dt,de):= dt\lambda(de)$ for $\mu$, such that $\{\tilde{\mu}=(\mu-\nu),[0,t]\times A\}_{t\geq0}$ is a martingale. Furthermore, the filtration $\mathbb{F}$ satisfies: $\mathcal{F}_0$ contains all $\mathbb{P}$-null sets in $\mathcal{F}$, and $\mathcal{F}_{t^+} := \bigcap_{\epsilon> 0}\mathcal{F}_{t+\epsilon}=\mathcal{F}_{t} $.

The differential form of Forward Stochastic Differential Equations (FSDEs) as shown below contains drift, diffusion, and jump components:
\begin{equation}\label{eqn:FSDE}
    dX_t = b \left(t, X_{t} \right) dt+ \sigma \left(t,X_{t}\right) dW_t + \int_{E}\beta\left(t,X_{t},e\right)\tilde{\mu}(dt,de).
\end{equation}
Then for the stochastic process $u(t,X_t)$, according to It\^{o}'s formula, we have
$$
    \begin{aligned}
        du(t,X_t)  
        = & - f\left(t,X_{t},u,\sigma^\top\nabla_x u,I_t\right)dt + \left\langle \sigma^\top \nabla_x u(t,X_{t}),dW_t \right\rangle      \\
                     & + \int_{E} \hat{u}_t(e) \tilde{\mu}(dt,de), \\
    \end{aligned}
$$
where $\hat{u}_t(e) := u\left(t,X_{t} + \beta (t,X_{t},e)\right) - u\left(t,X_{t}\right)$. Let 
$$
    \left\{
    \begin{aligned}
        Y_t          & = u(t,X_t),                                                      \\
        Z_t          & = \sigma^\top(t,X_{t}) \nabla_xu(t,X_{t}),                        \\
    \end{aligned}
    \right.
$$
we obtain the following Backward Stochastic Differential Equations (BSDEs):
\begin{equation}\label{eqn:BSDE}
    dY_t =- f\big(t,X_t,Y_t,Z_t,I_t \big)dt + Z_t^\top dW_t   + \int_E \hat{u}_t(e) \tilde{\mu}(dt,de) .
\end{equation}

Combining equations \eqref{eqn:FSDE} and \eqref{eqn:BSDE}, we ultimately transform the PIDEs \eqref{eqn:PIDE} into FBSDEJs:
\begin{equation}\label{eqn:FBSDEJ-diffform}
    \left\{
    \begin{aligned}
        dX_t =& b \left(t, X_{t} \right) dt+ \sigma \left(t,X_{t}\right) dW_t + \int_{E}\beta\left(t,X_{t},e\right)\tilde{\mu}(dt,de), \\
        dY_t =&- f\big(t,X_t,Y_t,Z_t,I_t \big)dt + Z_t^\top dW_t   + \int_E \hat{u}_t(e) \tilde{\mu}(dt,de),\\
    \end{aligned}
    \right.
\end{equation}
or in Integral form:
\begin{equation}\label{eqn:FBSDEJ}
    \left\{
    \begin{aligned}
        X_t & = X_0 +  \int_0^t\left( b \left(s, X_s \right) ds+ \sigma \left(s,X_s \right) dW_s + \int_{E}\beta\left(s,X_s,e\right)\tilde{\mu}(ds,de)\right), \\
        Y_t & = g(X_T) + \int_t^T \bigg(f\left(s,X_s,Y_s,Z_s,I_s\right)ds -   Z_s^\top dW_s \\
        &\qquad\qquad\qquad\qquad -  \int_{E} \left(u\left(t,X_s+\beta(t,X_s,e)\right) -u(t,X_s)\right) \tilde{\mu}(ds,de)\bigg).\\
    \end{aligned}
    \right.
\end{equation}
Since $X_t$ and $Y_t$ have different integration intervals with respect to the time variable $t$, they are referred to as forward and backward processes, respectively. The transformation of the PIDE \eqref{eqn:PIDE} into the FBSDEJ \eqref{eqn:FBSDEJ} is known as the nonlinear Feynman-Kac formula.

Let quadruple $(X^x_t,Y^x_t,Z^x_t,I^x_t)$ denote the solution of the FBSDEJ \eqref{eqn:FBSDEJ} with initial condition $X_0 = x$. According to 
Barles et al. \cite{barlesBackwardStochasticDifferential1997}
and 
Delong \cite{delongBackwardStochasticDifferential2013}
, a unique solution quadruple exists for the FBSDEJ \eqref{eqn:FBSDEJ} under Assumptions \ref{assumption:A1} - \ref{assumption:A5}, which will be introduced in section \ref{section:Preliminaries}. This can be summarized as the following lemma whose proof can be found in \cite{delongBackwardStochasticDifferential2013}.

\begin{lemma}[Theorem 4.1.3 and Proposition 4.1.1 in \cite{delongBackwardStochasticDifferential2013}]\label{lemma:existence_uniqueness}\label{lemma:FBSDEJ_PIDE}
    There exists a unique solution $(X^x_t,Y^x_t,Z^x_t,I^x_t)$ to the FBSDEJ \eqref{eqn:FBSDEJ} under Assumptions \textnormal{\ref{assumption:A1} - \ref{assumption:A5}}.
\end{lemma}


\subsection{Time dicretization of FBSDEJs}
In this subsection, we discretize the FBSDEJ \eqref{eqn:FBSDEJ} for the time variable and handle the integral terms.

Given a uniform partition $\mathbb{T}$ of the interval $[0,T]$, dividing it into $N$ sub-intervals with length $\Delta t = T/N$, we get the partition nodes $t_n = n\Delta t,~ n=0,1,\dots,N$. We consider the FBSDEJ \eqref{eqn:FBSDEJ} of the following form on each interval $[t_n,t_{n+1}]$:
\begin{equation}\label{eqn:FBSDEJ-discrete-form}
    \left\{
    \begin{aligned}
        X_t & = X_n +  \int_{t_n}^{t_{n+1}}\left( b \left(s, X_s \right) ds+ \sigma \left(s,X_s \right) dW_s + \int_{E}\beta\left(s,X_s,e\right)\tilde{\mu}(ds,de)\right), \\
        Y_t & = Y_n - \int_{t_n}^{t_{n+1}} \left(f\left(s,X_s,Y_s,Z_s,I_s\right)ds -   Z_s^\top dW_s -  \int_{E} \hat{u}_n(e) \tilde{\mu}(ds,de)\right).\\
    \end{aligned}
    \right.
\end{equation}

We discretize the FBSDEJ \eqref{eqn:FBSDEJ-discrete-form} using the Euler-Maruyama method:
\begin{equation}\label{eqn:discret-FBSDEJ}
    \left\{
    \begin{aligned}
        X_{n+1} =& X_n +  b \left(t_n, X_n \right) \Delta t + \sigma \left(t_n,X_n \right) \Delta W_n + \int_{E}\beta\left(t_n,X_n,e\right)\tilde{\mu}((t_n,t_{n+1}],de), \\
        Y_{n+1} =& Y_{n} - f\left(t_n,X_n,Y_n,Z_n,I_n\right) \Delta t + Z_n^\top \Delta W_n + \int_{E} \hat{u}_n(e) \tilde{\mu}((t_n,t_{n+1}],de).\\
    \end{aligned}
    \right.
\end{equation}
with
\begin{equation*}
    \left\{
    \begin{aligned}
        Y_n          & = u(t_n,X_n),                                                      \\
        Z_n          & = \sigma^\top(t_n,X_n) \nabla_xu(t_n,X_n) ,                        \\
        \hat{u}_n(e) & = u\big(t_n,X_n + \beta (t_n,X_{n},e)\big) - u\big(t_n,X_{n}\big). \\
    \end{aligned}
    \right.
\end{equation*}

Now we are going to discretize the non-local integral term in \eqref{eqn:discret-FBSDEJ}. The non-local integral terms have been extensively studied, and neural networks have been used to approximate them in \cite{luTemporalDifferenceLearning2023,castroDeepLearningSchemes2022,wangDeepLearningNumerical2023}. However, using neural networks for this approximation will make the training process more challenging, leading to difficulties in converging to satisfactory results. To address this issue, we introduce the Taylor expansion to simplify the integral terms based on the characteristics of the jump terms.

According to the definition of $\mu$, we have
\begin{equation}\label{eqn:integral-term-expansion}
    \begin{aligned}
        \int_{E}\beta\big(t,X_{t},e\big)\tilde{\mu}(dt,de) = & \int_{E}\beta\big(t,X_{t},e\big)\mu(dt,de) -  \int_{E}\beta\big(t,X_{t},e\big)\nu(dt,de)       \\
        =                                                     & \sum_{i=1}^{\mu_t} \beta\big(t,X_{t},e_i\big) - \int_{E}\beta\big(t,X_{t},e\big)\lambda(de) dt,
    \end{aligned}
\end{equation}
where $\mu _t$ represents the number of occurrences of the random event $e$ in the time interval $dt$. From \eqref{eqn:integral-term-expansion}, we can see that the integral term in the forward process in \eqref{eqn:discret-FBSDEJ} can be decomposed into one part related to the Poisson distribution and another integral term. When $\beta$ is given and integrable, the integral term in the forward stochastic differential equation \eqref{eqn:FSDE} can be directly calculated.

However, for the backward stochastic process in \eqref{eqn:discret-FBSDEJ}, since the solution $u$ is unknown, the integral term cannot be directly calculated and requires special treatment. Since automatic differentiation can be used in neural networks to handle the gradient of the network with respect to the input, we consider using the Taylor series to handle the integral term. By the Taylor expansion of $\hat{u}_n(e)$ at $x= X_n$, we obtain
\begin{equation*}
    u\big(t_n,X_n + \beta (t_n,X_{n},e)\big) - u\big(t_n,X_{n}\big) = \big\langle \nabla_x u(t_n,X_{n}) ,\beta (t_n,X_{n},e) \big\rangle + o\big(\beta (t_n,X_{n},e)\big),
\end{equation*}
where $o(\cdot)$ denotes higher-order infinitesimal terms. Integrating both sides of the above equation, we have
\begin{equation}\label{Taylor-integro}
    \int_{E} \hat{u}_n(e) \lambda(de) = \int_{E} \big\langle \nabla_x u(t_n,X_{n}) ,\beta (t_n,X_{n},e) \big\rangle\lambda(de) + R(t_n,X_n).
\end{equation}
Here, $R(t_n,X_n) \triangleq \int_{E} o\big(\beta (t_n,X_{n},e)\big) \lambda(de)$. When $\beta$ is sufficiently smooth and integrable, and the integration domain $E$ is appropriate, the order of integration and inner product can be exchanged. Since $\nabla_x u$ can be obtained numerically through automatic differentiation, by ignoring the remainder term in \eqref{Taylor-integro}, we get
\begin{equation}
    \int_{E} \hat{u}_n(e) \tilde{\mu}((t_n,t_{n+1}],de) \approx \sum_{i=1}^{\mu_n} \hat{u}_n(e_i) - \int_{E} \big\langle \nabla_x u(t_n,X_{n}) ,\beta (t_n,X_{n},e) \big\rangle\lambda(de) \Delta t,
\end{equation}
where $\mu_n$ represents the number of occurrences of the random event $e$ in the time interval $[t_n,t_{n+1}]$.

\begin{remark}
    The above approximation reuses the gradient terms in practical computations, which improves computational efficiency and does not incur additional computational overhead. 
\end{remark}

Combining the above analysis, we get the following discrete scheme
\begin{equation}\label{eqn:discret-FBSDE-final}
\left\{
\begin{aligned}
    X_{n+1} = X_n & +  b \left(t_n, X_n \right) \Delta t + \sigma \left(t_n,X_n \right) \Delta W_n                          \\
    & + \sum_{i=1}^{\mu_n} \beta\left(t_n,X_{n},e_i\right) - \int_{E}\beta\left(t_n,X_{n},e\right)\lambda(de) \Delta t\\
    Y_{n+1} = Y_{n} & - f\left(t_n,X_n,Y_n ,Z_n,I_n\right) \Delta t +   Z_n^\top \Delta W_n           \\
    & + \sum_{i=1}^{\mu_n} \hat{u}_n(e_i) - \int_{E} \left\langle \nabla_x u(t_n,X_{n}) ,\beta (t_n,X_{n},e)\right\rangle\lambda(de) \Delta t.
\end{aligned}
\right.
\end{equation}

\subsection{Forward-backward stochastic jump neural network}
To construct the loss function for the PIDEs \eqref{eqn:PIDE} based on the FBSDEJs \eqref{eqn:FBSDEJ} and implement the Forward-backward stochastic jump neural networks (FBSJNN, \cite{raissiForwardBackwardStochastic2018}) fitting format, we explore the following approach. The main idea is to use a neural network to fit $Y_n$, calculate the gradient terms and integral terms in the discrete format using automatic differentiation, and construct the loss function over the entire time interval to obtain the numerical approximation of $u(t,x)$.

Given a uniform partition $\mathbb{T}$ of the interval $[0,T]$ with step length $\Delta t$, we denote the time-discretized solution of the forward stochastic process by $\mathcal{X}_n$, which is approximated by Monte-Carlo method. Then using a neural network $\mathcal{N}$ to approximate the backward stochastic process $\mathcal{Y}_n$:
\begin{equation*}
    \mathcal{Y}_n =\mathcal{N}_{n}^{\theta} :=  \mathcal{N}(t_n, \mathcal{X}_n; \theta).
\end{equation*}
By automatic differentiation, we obtain the gradient terms $\mathcal{Z}_n$ by the following formula:
\begin{equation*}
    \mathcal{Z}_n = \sigma_n^\top\nabla_{x}\mathcal{Y}_n = \sigma^\top(t_n, \mathcal{X}_n)\nabla_{x}\mathcal{N}_{n}^{\theta}.
\end{equation*}
The integral terms $\mathcal{I}_n$ can be approximated by the numerical integration
\begin{equation*}
    \begin{aligned}
        \mathcal{I}_n =& \frac{1}{\Delta t}\sum_{i=1}^{\mu_n} \hat{\mathcal{N}}_n^{\theta}(e_i) - \int_{E} \left\langle \nabla_x \mathcal{N}_n^{\theta} ,\beta (t_n,X_{n},e)\right\rangle\lambda(de) \\
        =& \frac{1}{\Delta t}\sum_{i=1}^{\mu_n} \left(\mathcal{N}(t_n,X_n + \beta (t_n,X_{n},e);\theta) - \mathcal{N}_n^\theta\right) - \int_{E} \left\langle \nabla_x \mathcal{N}_n^{\theta} ,\beta \right\rangle\lambda(de) 
        .
    \end{aligned}
\end{equation*}

For simplicity, let us denote by $\mathcal{T}:[0,T]\times\mathbb{R}^d\times\mathbb{R}\times\mathbb{R}^d\times\mathbb{R}\mapsto \mathbb{R}$ the map for the backward stochastic process \eqref{eqn:discret-FBSDEJ}, such that
\begin{equation}
    \begin{aligned}\label{eqn:transfer-fun}
        Y_{n+1} &= \mathcal{T}\left(t_n,X_n,Y_n,Z_n,I_n\right) \\
                &= Y_{n}  - f\left(t_n,X_n,Y_n ,Z_n,I_n\right) \Delta t +   Z_n^\top \Delta W_n + I_n \Delta t.
    \end{aligned}
\end{equation}
Then the loss function $\mathcal{L}$ is defined as follows:
\begin{equation}\label{eqn:loss}
    \mathcal{L}(\theta) := \frac{1}{N+1}\left(\sum_{n=0}^{N-1}\mathbb{E}|\mathcal{Y}_{n+1} - \mathcal{T}\left(t_n,\mathcal{X}_n,\mathcal{Y}_n,\mathcal{Z}_n, \mathcal{I}_n\right)|^2 + \mathbb{E}|\mathcal{Y}_N-g(\mathcal{X}_N)|^2\right).
\end{equation}

To minimize the loss function $\mathcal{L}$, we can use the gradient descent method (or other optimization algorithm such as Adam \cite{kingmaAdamMethodStochastic2017}) to update the parameters of the neural network $\theta$. In this view, the PDE numerical solution problem is similar to a specific optimization problem with a designed loss function embedded with the PDE structure. With the optimal parameter $\theta^*$, we obtain the numerical solution $\mathcal{Y}_n(\theta^*)$ of the PIDE
\begin{equation}\label{eqn:optimal-para}
    \left\{
        \begin{aligned}
            \theta^* =& \arg\min_{\theta}\mathcal{L}(\theta), \\
            \mathcal{Y}_n(\theta^*) =& \mathcal{N}_{n}^{\theta^*} , ~~n=0,1,\dots,N.
        \end{aligned}
        \right.
\end{equation}

In this loss function, there are two main ways in which the equation structure is embedded:
\begin{itemize}
    \item Utilizing the recursion \eqref{eqn:transfer-fun} to recursively obtain the next value at node $n+1$ from the current neural network fitted value $\mathcal{Y}_n$ at node $n$, and computing the squared loss with the neural network fitted value $\mathcal{Y}_{n+1}$ at node $n+1$;
    \item For the terminal node $N$ (i.e., at the terminal time $T$), computing the squared loss between the neural network fitted value $\mathcal{Y}_N$ and the terminal condition $g(\mathcal{X}_N)$.
\end{itemize}

\begin{figure}[htbp]
    \centering
    \includegraphics[width=\textwidth, trim=2.5cm 3.2cm 2.5cm 3cm, clip]{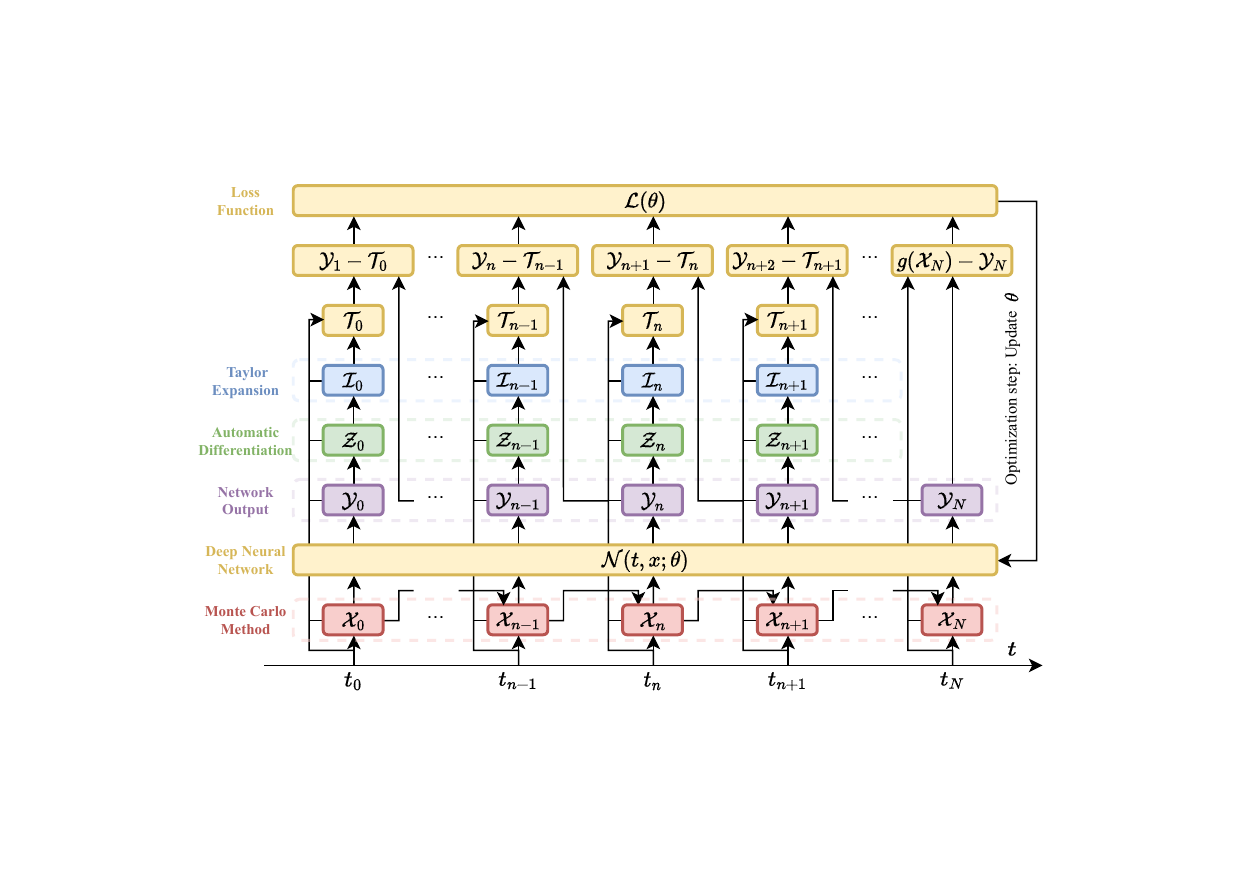}
    \caption{Forward-Backward Stochastic Jump Neural Network.}
    \label{fig:FBSJNN}
\end{figure}

All squared loss terms are averaged to obtain the final loss. Through the figure \ref{fig:FBSJNN} and pseudo-code, we can understand the construction process of the loss function and the algorithm more clearly.

\begin{algorithm}
    \caption{Forward-Backward Stochastic Jump Neural Network}
    \label{alg:FBSJNN}
    \KwIn{Given PIDE\eqref{eqn:PIDE}, time-equidistant discretization $\mathbb{T}$, deep neural network $\mathcal{N}(\cdot;\theta)$}


    Use the nonlinear Feynman-Kac formula to transform PIDE \eqref{eqn:PIDE} into FBSDEJ \eqref{eqn:FBSDEJ}\;

    Discretize FBSDEJ \eqref{eqn:FBSDEJ} by the proposed framework to obtain the discrete scheme \eqref{eqn:discret-FBSDE-final}\;
    \For{iter $<$ max iteration}{
        Simulate the forward stochastic process $\{\mathcal{X}_n\}$ using the Monte Carlo method\;

        Input $(t_n,\mathcal{X}_n)$ into the neural network $\mathcal{N}(\cdot;\theta)$ to obtain the output $\mathcal{Y}_n$, and use automatic differentiation to obtain the gradient $\mathcal{Z}_n$, then calculate the integral term $\mathcal{I}_n$ \;

        Calculate the loss function $\mathcal{L}$ according to \eqref{eqn:loss}, and optimize the neural network parameters $\theta$\;
    }
\end{algorithm}

Compared with existing works, such as Deep BSDE \cite{eDeepLearningBasedNumerical2017} or DBDP \cite{hureDeepBackwardSchemes2020}, the main innovation of FBSJNN framework proposed in this paper lies in the use of a smaller network and the handling of the integral term. For the integral term, we consider first expanding the integrand by using Taylor series, and thus simplify the integral calculation. We only use one network to approximate the solution of PIDEs, with the differential term being calculated by using automatic differentiation techniques, and the integral term calculated based on the differential term. Since only one network is used, compared to numerical methods that use separate neural networks to fit the differential and integral terms, the total parameter size used in the numerical method of this paper is smaller, which is more conducive to neural network parameter optimization.

We end this section with some remarks on the loss function \eqref{eqn:loss}.
\begin{remark}
    Why (${1}/{N+1}$)? Averaging is a common technique used in the field of machine learning. In the loss function of neural networks, taking the average over all samples can standardize the loss values, improve gradient stability, control the learning rate, and reduce the impact of noisy samples on model training. A more advanced approach is to use dynamic weighting to balance the sizes of different terms, which is an area where technical improvements can be made. 
\end{remark}
\begin{remark}
    Why $\mathbb{E}$? As $\mathcal{X}_n$ is generated by the Monte Carlo method, the expectation value of the loss function is the average loss value of the entire sample. This allows the training of the neural network to be conducted on batch data. This reflects the concept of “batch” in deep learning algorithms. Using batches has several benefits, including improved computational efficiency, better generalization by reducing over-fitting, and more stable gradient estimates during training.
\end{remark}
\begin{remark}
   When $\beta \equiv 0$, PIDEs simplify to PDEs, and the numerical methods and theoretical results developed in this paper remain applicable.
\end{remark}

\section{Consistency Analysis}
\label{section:Consistency Analysis}

In this section, we will analyze the consistency of the FBSJNN scheme.

\subsection{Preliminaries}\label{section:Preliminaries}

We will first present some useful notations, assumptions and theorems.

Unless otherwise specified, $K$ denotes a constant, and $|\cdot|$ represents either the absolute value or the $L^2$-norm.
Let us assume that:

\begin{enumerate}[label=(A\arabic*)]
    \item \label{assumption:A1} 
    Functions $b:[0,T]\times\mathbb{R}^d\to\mathbb{R}^d$ and $\sigma:[0,T]\times\mathbb{R}^d\to\mathbb{R}^{d\times d}$ satisfy Lipschitz condition. They are continuously differentiable ($C^1$) and their partial derivatives are bounded uniformly.
    
    \item \label{assumption:A2} 
    Measurable function $\beta(t,x,e) :[0,T]\times\mathbb{R}^d\times E\mapsto \mathbb{R}^d$ satisfies Lipschitz condition:
    \begin{equation*}
        \begin{aligned}
            |\beta(t,x,e)|                       & \leq K(1\wedge|e|),                                                  \\
            |\beta(t,x,e)-\beta(t,x^{\prime},e)| & \leq K|x-x^{\prime}|(1\wedge|e|).
        \end{aligned}
    \end{equation*}

    \item \label{assumption:A3} 
    Function $f$ satisfies Lipschitz condition with $\frac{1}{2}$-\text{H\"{o}lder} continuity on time domain, that is, for all $(t_{1},X_{1},Y_{1},Z_{1},I_{1})$ and $(t_{2},X_{2},Y_{2},Z_{2},I_{2})\in[0,T]\times \mathbb{R}^{d}\times \mathbb{R}\times \mathbb{R}^{d}\times \mathbb{R}$, there exist a constant $K$, such that $f$ satisfies:
    \begin{equation*}
        \begin{aligned}
            &|f(t_{2},X_{2},Y_{2},Z_{2},I_{2})-f(t_{1},X_{1},Y_{1},Z_{1},I_{1})| \\
            &\leq K (|t_{2}-t_{1}|^{1/2}+|X_{2}-X_{1}|+|Y_{2}-Y_{1}|+|Z_{2}-Z_{1}|+|I_{2}-I_{1}|).
        \end{aligned}
    \end{equation*}
    As a consequence, we have
    \begin{equation*}
        \sup_{0\leq t \leq T}|f(t,0,0,0,0)|<\infty.
    \end{equation*}
    Meanwhile, $f$ is continuously differentiable ($C^1$) and its partial derivatives are bounded uniformly.

    \item \label{assumption:A4} 
    L\'{e}vy measure $\lambda(de)$ satisfies the finite second moment condition:
    $$\int_E (1\wedge |e|^2) \lambda(de) <\infty.$$ 

    \item \label{assumption:A5} 
    Measurable functions $\delta:E\mapsto\mathbb{R}$ and $g:\mathbb{R}^d\mapsto\mathbb{R}$ satisfy
    \begin{equation*}
        \begin{aligned}
            |\delta(e)|                       & \leq K(1\wedge|e|),   \\
            |g(x)-g(x^{\prime})| & \leq K|x-x^{\prime}|.
        \end{aligned}
    \end{equation*}

\end{enumerate}


The following lemmas outline key properties of square-integrable martingales, which will be useful for subsequent analysis. Here, $L^2_W(\mathbb{R}^d)$ denotes the space of square-integrable processes that are adapted to a Brownian motion $W$ and take values in $\mathbb{R}^d$. Similarly, $L^2_{\tilde{\mu}}(\mathbb{R})$ represents the analogous space for processes adapted to the compensated Poisson measure $\tilde{\mu}$, with values in $\mathbb{R}$.
\begin{lemma}[Martingale Representation Theorem, \cite{situTheoryStochasticDifferential2005}]\label{lemma:MRT} For any square integrable martingale $M$, there exists $(Z,U)\in L^2_W(\mathbb{R}^d) \times L^2_{\tilde{\mu}}(\mathbb{R})$, such that for $t\in [0,T]$,
    $$
    M_t = M_0 + \int_{0}^{T} Z^\top dW_s + \int_{0}^{t} \int_E U \tilde{\mu}(ds,de).
    $$
\end{lemma}
\begin{lemma}[Conditional It\^{o} isometry, \cite{situTheoryStochasticDifferential2005}]\label{lemma:ConditionalIsometry} Let $A^1,A^2 \in L^2_W(\mathbb{R}^d)$, $B^1,B^2\in L^2_{\tilde{\mu}}(\mathbb{R})$, then:
    $$
    \begin{aligned}
        \mathbb{E} \left\{ \int_{t_n}^{t_{n+1}} A^1 dW_s \int_{t_n}^{t_{n+1}} A^2 dW_s \right\} =& \mathbb{E} \left\{ \int_{t_n}^{t_{n+1}} A^1  A^2 ds \right\}, \\
        \mathbb{E} \left\{ \int_{t_n}^{t_{n+1}} A^1 dW_s \int_{t_n}^{t_{n+1}} \int_E B^1 \tilde{\mu}(ds,de) \right\} =& 0,\\
        \mathbb{E} \left\{ \int_{t_n}^{t_{n+1}} \int_E B^1 \tilde{\mu}(ds,de) \int_{t_n}^{t_{n+1}} \int_E B^2 \tilde{\mu}(ds,de) \right\} =& \mathbb{E} \left\{ \int_{t_n}^{t_{n+1}} \int_E B^1 B^2 \tilde{\mu}(ds,de) \right\}.\\
    \end{aligned}
    $$
\end{lemma}

We also refer to the paper \cite{bouchardDiscretetimeApproximationDecoupled2008, castroDeepLearningSchemes2022} for the regularity properties of $Y$, $Z$ and $I$.
\begin{lemma}[$L^2$ regularity of $Y$, $Z$ and $I$, \cite{bouchardDiscretetimeApproximationDecoupled2008, castroDeepLearningSchemes2022}]\label{lemma:L2regularityItem} Under assumptions \ref{assumption:A1} - \ref{assumption:A5}, there exists a constant $K > 0$ such that,
    $$\epsilon^Y(\Delta t) \leq K\Delta t, \quad\epsilon^Z(\Delta t) \leq K\Delta t, \quad\text{and}\quad  \epsilon^I(\Delta t) \leq K\Delta t,$$
    where
    \begin{equation}
        \begin{aligned}
            \epsilon^Y(\Delta t) :=& \mathbb{E}\bigg\{\sum_{n=0}^{N-1}\int_{t_n}^{t_{n+1}} |Y_s - Y_{n}|^2 ds \bigg\},\\
            \epsilon^Z(\Delta t) :=& \mathbb{E}\bigg\{\sum_{n=0}^{N-1}\int_{t_n}^{t_{n+1}} |Z_s - \bar{Z}_{n}|^2 ds \bigg\}, &\bar{Z}_{n} :=& \frac{1}{\Delta t} \mathbb{E}_n \big\{ \int_{t_n}^{t_{n+1}}Z_s ds \big\},\\
            \epsilon^I(\Delta t) :=& \mathbb{E}\bigg\{\sum_{n=0}^{N-1}\int_{t_n}^{t_{n+1}} |I_s - \bar{I}_{n}|^2 ds \bigg\}, &\bar{I}_{n} :=& \frac{1}{\Delta t} \mathbb{E}_n \big\{ \int_{t_n}^{t_{n+1}}I_s ds \big\}.
        \end{aligned}
    \end{equation}
\end{lemma}

Based on \ref{assumption:A5}, one can prove the square-integral property of $f$; See \cite{castroDeepLearningSchemes2022} for more details.
\begin{lemma}[Square-integral property of $f$, \cite{castroDeepLearningSchemes2022}]\label{lemma:fSquareIntegral} Under assumption \textnormal{\ref{assumption:A5}}, for the unique solution $(X_t, Y_t, Z_t, I_t)$ to the FBSDEJs \eqref{eqn:FBSDEJ}, there exists a constant $K > 0$ such that,
    $$\mathbb{E} \bigg\{\int_{0}^{T} f^2 (s, X_s, Y_s, Z_s, I_s) ds\bigg\} < K.$$
\end{lemma}

The fundamental result of the series of works by Hornik, Stinchcombe, and White \cite{hornikMultilayerFeedforwardNetworks1989,stinchcombeUniversalApproximationUsing1989} establishes the approximate capabilities of neural networks. 

\begin{lemma}[Universal Approximation Theorem, \cite{hornikMultilayerFeedforwardNetworks1989, stinchcombeUniversalApproximationUsing1989}]\label{lemma:UAT} Let $\kappa$ be a non-constant activation function that is a $C^k$ function. Then, a neural network with a single hidden layer, a sufficient number of neurons, and the activation function $\kappa$ can approximate any continuous function and its derivatives up to order $k$ arbitrarily well on any compact subset of $\mathbb{R}^d$. Here, $d$ represents the input dimension of the neural network, while the output dimension can be any positive integer.
\end{lemma}

\subsection{Consistency of the FBSJNN} Following the paper \cite{hureDeepBackwardSchemes2020, castroDeepLearningSchemes2022, bouchardDiscretetimeApproximationDecoupled2008}, we are going to derive the consistency of the proposed FBSJNN framework.

\begin{theorem}[Consistency of the FBSJNN]\label{thm:Consistency of the FBSJNN} Under assumptions \ref{assumption:A1} - \ref{assumption:A5}, the FBSJNN scheme converges to the solution of the FBSDEJ \eqref{eqn:FBSDEJ} and there exists a constant $K > 0$, independent of $\mathbb{T}$, such that
    \begin{equation}\label{eqn:Consistency of the FBSJNN}
        \begin{aligned}
            \max_{n} \big(\mathbb{E} \{|Y_n - \mathcal{Y}_n(\theta^*)|^2\} +& \mathbb{E} \{|Z_n - \mathcal{Z}_n(\theta^*)|^2\} + \mathbb{E} \big\{| I_n-\mathcal{I}_n(\theta^*)|^2\big\}\big)\\
            &\leq K \big(\Delta t +  \epsilon^Y + \epsilon^Z+ \epsilon^I + (N+1)\epsilon^{\mathcal{Y}} \big),
        \end{aligned}
    \end{equation}
where $\epsilon^{\mathcal{Y}}$ is the neural network approximation error
\begin{equation*}
    \epsilon^{\mathcal{Y}} := \inf_{m\in\mathbb{Z}^*} \max_n \mathbb{E}\{|Y_n-\mathcal{Y}_n|^2\},
\end{equation*}
and $m$ denotes the number of the neurons in a single hidden layer neural network.
\end{theorem}

\begin{proof} To prove the theorem, we first define the following $\mathcal{F}$-adapted discrete auxiliary process
\begin{equation}\label{eq3.3}
    \left\{
    \begin{aligned}
        \hat{Y}_n &= \mathbb{E}_n \{\mathcal{Y}_{n+1} \} + f(t_n, \mathcal{X}_n, \hat{Y}_n, \hat{Z}_n, \hat{I}_n) \Delta t,\\
        \hat{Z}_n &=  \frac{1}{\Delta t} \mathbb{E}_n \{\mathcal{Y}_{n+1} \Delta W_n\},\\
        \hat{I}_n &=  \frac{1}{\Delta t} \mathbb{E}_n \left\{\mathcal{Y}_{n+1} \int_E \tilde{\mu}((t_n,t_{n+1}],de) \right\}.\\
    \end{aligned}
    \right.
\end{equation}
where $\mathbb{E}_n\{\cdot\} := \mathbb{E}\{\cdot| \mathcal{F}_{t_n}\}$ is the conditional expectation given $\mathcal{F}_{t_n}$.

To avoid confusion, we will unify five sets of notation here. With index $t$ refers to time variable and $n$ refers to the index of the time discretization.

1. $(X, Y, Z, I)$ denote the exact solution;

2. $(\mathcal{X}, \mathcal{Y}, \mathcal{Z}, \mathcal{I})$ denote the approximate solution;

3. $(\hat{Y}, \hat{Z}, \hat{I})$ denote the auxiliary process;

4. $(\bar{Z}, \bar{I})$ represent the regularity of $Z$ and $I$, see Lemma \ref{lemma:L2regularityItem} for details;

5. $(Z^*, U^*)$ will be introduced later when applying Lemma \ref{lemma:MRT}.

The proof of the theorem is structured into four distinct steps. 
In the first step, we estimate the local error between the exact solution and the auxiliary process, given by $\mathbb{E} \{|Y_n - \hat{Y}_n|^2 \}$. 
Next, in the second step, we derive an upper bound for the error between the exact solution and the approximate solution, expressed as $\max_{n} \mathbb{E} \{|Y_n - \mathcal{Y}_n|^2\}$, using Gronwall's lemma. 
In the third step, we analyze the optimization error introduced by the structure of the loss function, which leads to different optimization parameters. We then estimate the error in the numerical solution of the optimal parameter, $\max_{n} \mathbb{E} \{|Y_n - \mathcal{Y}_n(\theta^*)|^2\}$.
Finally, in the fourth step, we introduce the neural network approximation error and deduce the result of the theorem.

{\bf Step 1:}
Reviewing the expression of BSDE in  \eqref{eqn:FBSDEJ-discrete-form}, by taking the conditional expectation $\mathbb{E}_n\{\cdot\}$ on both sides of the equation, we obtain:
\begin{equation*}
    \mathbb{E}_n \{Y_{n+1} \} = Y_n - \mathbb{E}_n \left\{ \int_{t_n}^{t_{n+1}}f(s, X_s, Y_s, Z_s, I_s) ds \right\}.
\end{equation*}
Subtracting the above equation from the first equation in (\ref{eq3.3}) yields
\begin{equation*}
    \begin{aligned}
        Y_n - \hat{Y}_n =& \mathbb{E}_n \{Y_{n+1} - \mathcal{Y}_{n+1}\} \\
        &+ \mathbb{E}_n \bigg\{ \int_{t_n}^{t_{n+1}} \bigg(f(s, X_s, Y_s, Z_s, I_s) -f(t_n, \mathcal{X}_n, \hat{Y}_n, \hat{Z}_n, \hat{I}_n)\bigg) ds \bigg\}.
    \end{aligned}
\end{equation*}
Using Young inequality and Jensen inequality, we have
\begin{equation}\label{eqn:proof_y_n_hat_y_n}
    \begin{aligned}
        |Y_n - \hat{Y}_n|^2 \leq & (1 + K \Delta t) |\mathbb{E}_n \{Y_{n+1} - \mathcal{Y}_{n+1}\}|^2 +\left(1 + \frac{1}{K \Delta t}\right)\cdot\\
        & \left|\mathbb{E}_n \bigg\{ \int_{t_n}^{t_{n+1}} \bigg(f(s, X_s, Y_s, Z_s, I_s) -f(t_n, \mathcal{X}_n, \hat{Y}_n, \hat{Z}_n, \hat{I}_n)\bigg) ds \bigg\}\right|^2\\
        \leq & (1 + K \Delta t) \mathbb{E}_n \{|Y_{n+1} - \mathcal{Y}_{n+1}|^2 \}+\left(1 + \frac{1}{K \Delta t}\right)\cdot\\
        & \;\;\mathbb{E}_n \bigg\{ \left|\int_{t_n}^{t_{n+1}} \bigg(f(s, X_s, Y_s, Z_s, I_s) -f(t_n, \mathcal{X}_n, \hat{Y}_n, \hat{Z}_n, \hat{I}_n)\bigg) ds \right|^2\bigg\}.\\
    \end{aligned}
\end{equation}
For the second term in right hand side, applying Cauchy-Schwarz inequality and Lipschitz condition on $f$, we have
\begin{equation}
    \label{eqn:Lipschitz_f}
    \begin{aligned}
        &\mathbb{E}\bigg\{\mathbb{E}_n \bigg\{ \left|\int_{t_n}^{t_{n+1}} \bigg(f(s, X_s, Y_s, Z_s, I_s) -f(t_n, \mathcal{X}_n, \hat{Y}_n, \hat{Z}_n, \hat{I}_n)\bigg) ds \right|^2\bigg\}\bigg\} \\
        &\leq \mathbb{E} \bigg\{ \Delta t \int_{t_n}^{t_{n+1}} \bigg(f(s, X_s, Y_s, Z_s, I_s) -f(t_n, \mathcal{X}_n, \hat{Y}_n, \hat{Z}_n, \hat{I}_n)\bigg)^2 ds \bigg\}\\
        &\leq 5 \Delta t K^2 \mathbb{E} \bigg\{ \int_{t_n}^{t_{n+1}} \big( |s-t_n|+ |X_s-\mathcal{X}_n|^2 + |Y_s-\hat{Y}_n|^2 +|Z_s-\hat{Z}_n|^2 +|I_s-\hat{I}_n|^2 \big) ds \bigg\} .
    \end{aligned}
\end{equation}

Note that all the terms on the right-hand side of \eqref{eqn:Lipschitz_f} can be estimated as follows.
For the first term, it is easy to get
\begin{equation}\label{eqn:step1term1}
    \mathbb{E}\left\{\int_{t_n}^{t_{n+1}} |s-t_n| ds\right\} =( \frac{1}{2}s^2 -t_n s) |_{t_n}^{t_{n+1}} = \frac{1}{2} \Delta t^2.
\end{equation}
For the second term, according Bouchard and Elie (2008) \cite{bouchardDiscretetimeApproximationDecoupled2008}, we can obtain
\begin{equation}
    \mathbb{E}\left\{\int_{t_n}^{t_{n+1}} |X_s-\mathcal{X}_n|^2 ds \right\}\leq \Delta t\mathbb{E}\left\{\sup_{t_n\leq s\leq t_{n+1}}|X_s-\mathcal{X}_n|^2 \right\} \leq K \Delta t^2.
\end{equation}
The third term can be bounded by
\begin{equation}
    \begin{aligned}
        \mathbb{E}\left\{\int_{t_n}^{t_{n+1}} |Y_s-\hat{Y}_n|^2 ds \right\}\leq& \mathbb{E}\left\{\int_{t_n}^{t_{n+1}}( 2|Y_s - Y_n|^2+ 2|Y_n-\hat{Y}_n|^2)ds\right\}\\
        \leq &2 \mathbb{E}\left\{ \int_{t_n}^{t_{n+1}} |Y_s-Y_n|^2 ds\right\} + 2\Delta t \mathbb{E}\left\{|Y_n - \hat{Y}_n|^2\right\}.
    \end{aligned}
\end{equation}
To estimate the fourth term, we employ
$$
\begin{aligned}
    \mathbb{E}\left\{\int_{t_n}^{t_{n+1}} |Z_s-\hat{Z}_n|^2 ds\right\} =& \mathbb{E}\left\{\int_{t_n}^{t_{n+1}} |Z_s-\bar{Z}_n + \bar{Z}_n -\hat{Z}_n|^2 ds\right\}\\
    = & \mathbb{E}\left\{\int_{t_n}^{t_{n+1}} |Z_s-\bar{Z}_n|^2 ds\right\}
        + \Delta t \mathbb{E}\left\{|\bar{Z}_n -\hat{Z}_n|^2\right\}.
\end{aligned}
$$
Recalling the expression of BSDEs,
$$
Y_{n+1} = Y_{n} - \int_{t_n}^{t_{n+1}} \bigg(f_s ds -   Z_s^\top dW_s -  \int_{E} \hat{u}_s(e) \tilde{\mu}(ds,de)\bigg),
$$
where $f_s := f(s, X_s, Y_s, Z_s, I_s)$, multiplying both sides of the equation by $\Delta W_n$ and taking the conditional expectation $\mathbb{E}_n$, we then get
$$
\begin{aligned}
    \mathbb{E}_n\left\{\Delta W_n Y_{n+1}\right\} =& - \mathbb{E}_n\left\{\Delta W_n \int_{t_n}^{t_{n+1}} f_s ds\right\}+ \mathbb{E}_n\left\{\Delta W_n \int_{t_n}^{t_{n+1}} Z_s^\top dW_s\right\} \\
    &  + \mathbb{E}_n\left\{\Delta W_n\int_{t_n}^{t_{n+1}}\int_{E} \hat{u}_s(e) \tilde{\mu}(ds,de)\right\}.
\end{aligned}
$$
Using conditional It\^{o} isometry Lemma \ref{lemma:ConditionalIsometry} leads to
$$
\Delta t \bar{Z}_n = \mathbb{E}_n\{\Delta W_n Y_{n+1}\}+ \mathbb{E}_n\left\{\Delta W_n \int_{t_n}^{t_{n+1}} f_s ds\right\}.
$$
Employing the definition of $\hat{Z}_n$, taking it into account that $\mathbb{E}_n\{\Delta W_n \mathbb{E}_n\{Y_{n+1}-\mathcal{Y}_{n+1}\}\} =0$, we obtain
$$
\begin{aligned}
    \Delta t (\bar{Z}_n - \hat{Z}_n) = & \mathbb{E}_n\left\{\Delta W_n (Y_{n+1}-\mathcal{Y}_{n+1})\right\}+ \mathbb{E}_n\left\{\Delta W_n \int_{t_n}^{t_{n+1}} f_sds\right\}\\
    = & \mathbb{E}_n\left\{\Delta W_n\big( Y_{n+1}-\mathcal{Y}_{n+1} - \mathbb{E}_n\{Y_{n+1}-\mathcal{Y}_{n+1}\}\big)\right\}\\
    & + \mathbb{E}_n\left\{\Delta W_n \int_{t_n}^{t_{n+1}} f_s ds\right\}. \\
\end{aligned}
$$
Using Cauchy-Schwarz inequality, we have
$$
\begin{aligned}
    &\Delta t^2 |\bar{Z}_n - \hat{Z}_n|^2 \\
    & \leq 2d \mathbb{E}_n^2\bigg\{\Delta W_n \big( Y_{n+1}-\mathcal{Y}_{n+1} - \mathbb{E}_n\{Y_{n+1}-\mathcal{Y}_{n+1}\}\big)\bigg\} + 2d\mathbb{E}_n^2\bigg\{\Delta W_n \int_{t_n}^{t_{n+1}} f_s ds\bigg\} \\
    & \leq  2d\Delta t \mathbb{E}_n\bigg\{\big((Y_{n+1}-\mathcal{Y}_{n+1}) - \mathbb{E}_n\{Y_{n+1}-\mathcal{Y}_{n+1}\}\big)^2 + \big(\int_{t_n}^{t_{n+1}} f_s ds\big)^2\bigg\}\\
    & \leq  2d\Delta t \mathbb{E}_n\bigg\{(Y_{n+1}-\mathcal{Y}_{n+1})^2 + \mathbb{E}_n^2\{Y_{n+1}-\mathcal{Y}_{n+1}\} \\
    & \qquad\qquad\qquad-2(Y_{n+1}-\mathcal{Y}_{n+1})\mathbb{E}_n\{Y_{n+1}-\mathcal{Y}_{n+1}\}+\Delta t \int_{t_n}^{t_{n+1}} f^2_s ds\bigg\}\\
    &\leq  2d\Delta t \mathbb{E}_n\bigg\{(Y_{n+1}-\mathcal{Y}_{n+1})^2 - \mathbb{E}_n^2\{Y_{n+1} - \mathcal{Y}_{n+1}\} + \Delta t \int_{t_n}^{t_{n+1}} f^2_s ds\bigg\}.\\
\end{aligned}
$$
Taking the expectation $\mathbb{E}$ both side, we get the estimate for the fourth term
\begin{equation}
    \begin{aligned}
        &\mathbb{E}\left\{\int_{t_n}^{t_{n+1}} |Z_s-\hat{Z}_n|^2 ds\right\} \leq  \mathbb{E} \left\{\int_{t_n}^{t_{n+1}} |Z_s-\bar{Z}_n|^2 ds\right\}\\
        & \qquad + 2d \mathbb{E} \bigg\{ (Y_{n+1} - \mathcal{Y}_{n+1})^2 - \mathbb{E}_n^2\{Y_{n+1} - \mathcal{Y}_{n+1}\} + \Delta t\int_{t_n}^{t_{n+1}} f^2_s ds\bigg\} .\\
    \end{aligned}
\end{equation}
For the fifth term, we have similar calculation as the fourth term. First, we have
$$
\mathbb{E}\left\{\int_{t_n}^{t_{n+1}} |I_s-\hat{I}_n|^2 ds\right\}
    =  \mathbb{E}\left\{\int_{t_n}^{t_{n+1}} |I_s-\bar{I}_n|^2 ds\right\}
        + \Delta t \mathbb{E}\left\{|\bar{I}_n -\hat{I}_n|^2\right\}.
$$
Multiplying both sides of BSDEs in \eqref{eqn:discret-FBSDEJ} by $\int_E \tilde{\mu}((t_n, t_{n+1}],de)$, taking the conditional expectation $\mathbb{E}_n$, and using conditional It\^{o} isometry Lemma \ref{lemma:ConditionalIsometry}, we obtain
$$
\Delta t\bar{I}_n = \mathbb{E}_n\left\{\int_E \tilde{\mu}((t_n, t_{n+1}],de) Y_{n+1}\right\}+ \mathbb{E}_n\left\{\int_E \tilde{\mu}((t_n, t_{n+1}],de) \int_{t_n}^{t_{n+1}} f_s ds\right\}.
$$
In view of the definition of $\hat{I}_n$, there have
$$
\begin{aligned}
    &\Delta t^2|\bar{I}_n - \hat{I}_n|^2 \\
    & \leq 2 \mathbb{E}_n^2\bigg\{\int_E \tilde{\mu}((t_n, t_{n+1}],de) \big( Y_{n+1}-\mathcal{Y}_{n+1} - \mathbb{E}_n\{Y_{n+1}-\mathcal{Y}_{n+1}\}\big)\bigg\} \\
    & \quad+ 2\mathbb{E}_n^2\bigg\{\int_E \tilde{\mu}((t_n, t_{n+1}],de) \int_{t_n}^{t_{n+1}} f_s ds\bigg\} \\
    & \leq  2\big(\Delta t\int_E \lambda (de)\big)^2 \mathbb{E}_n\bigg\{\big((Y_{n+1}-\mathcal{Y}_{n+1}) - \mathbb{E}_n\{Y_{n+1}-\mathcal{Y}_{n+1}\}\big)^2 + \big(\int_{t_n}^{t_{n+1}} f_s ds\big)^2\bigg\}\\
    &\leq  2\big(\Delta t\int_E \lambda (de)\big)^2 \mathbb{E}_n\bigg\{(Y_{n+1}-\mathcal{Y}_{n+1})^2 - \mathbb{E}_n^2\{Y_{n+1} - \mathcal{Y}_{n+1}\} + \Delta t \int_{t_n}^{t_{n+1}} f^2_s ds\bigg\}.\\
\end{aligned}
$$
Taking the expectation $\mathbb{E}$ both sides leads to the estimate for the fifth term
\begin{equation}\label{eqn:step1term5}
    \begin{aligned}
        &\mathbb{E}\int_{t_n}^{t_{n+1}} |I_s-\hat{I}_n|^2 ds \leq  \mathbb{E} \int_{t_n}^{t_{n+1}} |I_s-\bar{I}_n|^2 ds\\
        & \; + 2\big(\int_E \lambda (de)\big)^2\Delta t \mathbb{E} \bigg\{ (Y_{n+1} - \mathcal{Y}_{n+1})^2 - \mathbb{E}_n^2\{Y_{n+1} - \mathcal{Y}_{n+1}\} + \Delta t\int_{t_n}^{t_{n+1}} f^2_s ds\bigg\} .\\
    \end{aligned}
\end{equation}

Now, by taking the expectation $\mathbb{E}\{\cdot\}$ on both sides of the inequality \eqref{eqn:proof_y_n_hat_y_n}, and applying \eqref{eqn:Lipschitz_f}–\eqref{eqn:step1term5}, we can select an appropriate constant $K$, leading to the following estimate
\begin{equation}
    \begin{aligned}\label{eqn:estimate2_y_n_hat_y_n}
        &\mathbb{E} \left\{|Y_n - \hat{Y}_n|^2 \right\}\\
        &\leq K \bigg(\Delta t ^2  + \mathbb{E} \left\{\int_{t_n}^{t_{n+1}} |Y_s-Y_n|^2 ds\right\} + (1+\Delta t)  \mathbb{E}\left\{|Y_{n+1} - \mathcal{Y}_{n+1}|^2\right\}\\ 
        &\qquad + \mathbb{E} \left\{\int_{t_n}^{t_{n+1}} |Z_s-\bar{Z}_n|^2 ds \right\}+ \mathbb{E} \left\{\int_{t_n}^{t_{n+1}} |I_s-\bar{I}_n|^2 ds\right\} + \Delta t\mathbb{E}\left\{ \int_{t_n}^{t_{n+1}} f^2_s ds\right\}\bigg).
    \end{aligned}
\end{equation}
Summing $\mathbb{E} \{|Y_n - \hat{Y}_n|^2 \}$ from $0$ to $N-1$, and applying Lemmas \ref{lemma:L2regularityItem} and \ref{lemma:fSquareIntegral}, we obtain
\begin{equation}
    \begin{aligned}\label{eqn:estimate_y_n_hat_y_n_sum_condition}
        &\sum_{n=0}^{N-1}\mathbb{E} \left\{|Y_n - \hat{Y}_n|^2 \right\}\leq K \bigg(\Delta t + \epsilon^Y+ \epsilon^Z  + \epsilon^I  + (1+\Delta t)  \sum_{n=1}^{N}\mathbb{E}\left\{|Y_{n} - \mathcal{Y}_{n}|^2\right\}\bigg).\\
    \end{aligned}
\end{equation}

{\bf Step 2:} In this step. we aim to estimate $\mathbb{E}\{|Y_n - \mathcal{Y}_n|^2\}$. Using Young inequality, it is obvious that for any $\gamma\in (0,1)$,
\begin{equation*}
    \begin{aligned}
        \mathbb{E}\left\{ |Y_n - \hat{Y}_n|^2\right\} =& \mathbb{E} \left\{ |Y_n -\mathcal{Y}_n+\mathcal{Y}_n- \hat{Y}_n |^2 \right\}\\
        \geq & (1-\gamma)\mathbb{E} \left\{|Y_n - \mathcal{Y}_n|^2\right\} - \frac{1}{\gamma}\mathbb{E}\left\{|\mathcal{Y}_n - \hat{Y}_n|^2\right\}.
    \end{aligned}
\end{equation*}
Therefore, by using estimate \eqref{eqn:estimate2_y_n_hat_y_n}, 
\begin{equation*}
    \begin{aligned}
        &\mathbb{E} \left\{|Y_n - \mathcal{Y}_n|^2\right\} \leq \frac{1}{1-\gamma}\mathbb{E}\left\{ |Y_n - \hat{Y}_n|^2\right\} + \frac{1}{(1-\gamma)\gamma}\mathbb{E}\left\{|\mathcal{Y}_n - \hat{Y}_n|^2\right\}\\
        &\leq K \bigg(\Delta t ^2  + \mathbb{E} \left\{\int_{t_n}^{t_{n+1}} |Y_s-Y_n|^2 ds\right\} + (1+\Delta t)  \mathbb{E}\left\{|Y_{n+1} - \mathcal{Y}_{n+1}|^2\right\}\\ 
        &\qquad \qquad \;\; + \mathbb{E} \left\{\int_{t_n}^{t_{n+1}} |Z_s-\bar{Z}_n|^2 ds \right\}+ \mathbb{E} \left\{\int_{t_n}^{t_{n+1}} |I_s-\bar{I}_n|^2 ds\right\} \\
        &\qquad \qquad \;\; + \Delta t\mathbb{E}\left\{ \int_{t_n}^{t_{n+1}} f^2_s ds\right\}+ \mathbb{E}\{|\mathcal{Y}_n - \hat{Y}_n|^2\}\bigg),
    \end{aligned}
\end{equation*}
which is clearly a iterative process about $\mathbb{E} \{|Y_n - \mathcal{Y}_n|^2\}$. Applying the Gronwall's lemma, recalling the terminal condition, Lemma \ref{lemma:L2regularityItem} and Lemma 4.7 in \cite{castroDeepLearningSchemes2022}, we have
\begin{equation}
    \begin{aligned}\label{eqn:max_Y_n_cal_Y_n}
        &\max_{n} \mathbb{E} \left\{|Y_n - \mathcal{Y}_n|^2\right\} \\
        &\qquad\leq K \bigg(\Delta t + \epsilon^Y+ \epsilon^Z  + \epsilon^I + \mathbb{E}\left\{|g(\mathcal{X}_N)-\mathcal{Y}_N|^2\right\}+ \sum_{n=1}^{N-1} \mathbb{E} \left\{|\mathcal{Y}_n-\hat{Y}_n|^2\right\}\bigg).\\
    \end{aligned}
\end{equation}

{\bf Step 3:}
Recalling the definition of $\hat{Y}_n$, according to Lemma \ref{lemma:MRT}, there exist $Z_t^*\in L^2(W)$ and $U_t^* \in L^2(\tilde{\mu})$, such that
\begin{equation*}
    \begin{aligned}
        \mathcal{Y}_{n+1} =& \mathbb{E}_n\{\mathcal{Y}_{n+1}\} + \int_{t_n}^{t_{n+1}} Z_s^{*\top} dW_s + \int_{t_n}^{t_{n+1}} \int_{E} U_s^* \tilde{\mu}(ds,de),\\
            = & \hat{Y}_n -  f(t_n, \mathcal{X}_n, \hat{Y}_n, \hat{Z}_n, \hat{I}_n) \Delta t + \int_{t_n}^{t_{n+1}} Z_s^{*\top} dW_s + \int_{t_n}^{t_{n+1}} \int_{E} U_s^* \tilde{\mu}(ds,de).\\
    \end{aligned}
\end{equation*}
Let us consider the one-step error in backward recursion \eqref{eqn:transfer-fun},
\begin{equation*}
    \begin{aligned}
        &\mathcal{Y}_{n+1} - \mathcal{T}(t_n,\mathcal{X}_n, \mathcal{Y}_n, \mathcal{Z}_n, \mathcal{I}_n)\\
        &=\hat{Y}_n -\mathcal{Y}_n + \big(f(t_n,\mathcal{X}_n, \mathcal{Y}_n, \mathcal{Z}_n, \mathcal{I}_n) - f(t_n, \mathcal{X}_n, \hat{Y}_n, \hat{Z}_n, \hat{I}_n)\big)  \Delta t\\
        &\quad + \int_{t_n}^{t_{n+1}} \left(Z_s^{*\top}-\mathcal{Z}_n^\top\right) dW_s + \int_{t_n}^{t_{n+1}} \left(\int_{E} U_s^* \tilde{\mu}(ds,de)-\mathcal{I}_nds\right) .\\
    \end{aligned}
\end{equation*}
Noting the decomposition of $(Z_s^{*\top}-\mathcal{Z}_n^\top)$ and $\left(\int_{E} U_s^* \tilde{\mu}(ds,de)-\mathcal{I}_nds\right)$, one can find that
\begin{equation*}
    \begin{aligned}
        &\mathcal{Y}_{n+1} - \mathcal{T}(t_n,\mathcal{X}_n, \mathcal{Y}_n, \mathcal{Z}_n, \mathcal{I}_n)\\
        &=\hat{Y}_n -\mathcal{Y}_n + \big(f(t_n,\mathcal{X}_n, \mathcal{Y}_n, \mathcal{Z}_n, \mathcal{I}_n) - f(t_n, \mathcal{X}_n, \hat{Y}_n, \hat{Z}_n, \hat{I}_n)\big)  \Delta t\\
        &\quad + \int_{t_n}^{t_{n+1}} \left(\hat{Z}^{\top}_n-\mathcal{Z}_n^\top\right) dW_s + \int_{t_n}^{t_{n+1}} \left( \hat{I}_n -\mathcal{I}_n\right)ds \\
        &\quad + \int_{t_n}^{t_{n+1}} \left(Z_s^{*\top}-\hat{Z}_n^\top\right) dW_s + \int_{t_n}^{t_{n+1}} \left(\int_{E} U_s^* \tilde{\mu}(ds,de)-\hat{I}_nds\right) .\\
    \end{aligned}
\end{equation*}
Employing the definition of It\^{o} integral and the compensated Poisson measure, we have
$$\mathbb{E}\left\{\int_{t_n}^{t_{n+1}} \left(Z_s^{*\top}-\hat{Z}_n^\top\right) dW_s\right\} =\mathbb{E}\left\{\int_{t_n}^{t_{n+1}} \left(\int_{E} U_s^* \tilde{\mu}(ds,de)-\hat{I}_nds\right) \right\}= 0,$$
and therefore,
\begin{equation}\label{eqn:loss_decomposition}
    \begin{aligned}
        &\mathbb{E}\left\{|\mathcal{Y}_{n+1} - \mathcal{T}(t_n,\mathcal{X}_n, \mathcal{Y}_n, \mathcal{Z}_n, \mathcal{I}_n)|^2 \right\}\\
        &= \mathbb{E}\left\{|\hat{Y}_n -\mathcal{Y}_n + \big(f(t_n,\mathcal{X}_n, \mathcal{Y}_n, \mathcal{Z}_n, \mathcal{I}_n) - f(t_n, \mathcal{X}_n, \hat{Y}_n, \hat{Z}_n, \hat{I}_n)\big)  \Delta t|^2\right\}\\
        &\quad + \Delta t \bigg( \mathbb{E}\left\{|\hat{Z}_n - \mathcal{Z}_n|^2+| \hat{I}_n-\mathcal{I}_n|^2\right\}\bigg)\\
        &\quad + \mathbb{E} \left\{|\int_{t_n}^{t_{n+1}}(Z_s^* - \hat{Z}_n)dW_s|^2\right\} + \mathbb{E} \left\{|\int_{t_n}^{t_{n+1}}(\int_{E} U_s^* \tilde{\mu}(ds,de) - \hat{I}_nds)|^2 \right\}.
    \end{aligned}
\end{equation}
Let $\mathcal{L}^{(1)}_n$ represent the first two terms of the above equation, and $\mathcal{L}^{(2)}_n$ represent the remaining terms. Since the last two terms are independent of the neural network parameters $\theta$, they do not influence the optimization process. However, $\mathcal{L}^{(1)}_n$ clearly depends on $\theta$, and thus affects the optimization. Then $\mathcal{L}^{(1)}_n$ can be denoted by
\begin{equation*}
    \begin{aligned}
        \mathcal{L}^{(1)}_n(\theta) =& \mathbb{E}\left\{|\hat{Y}_n -\mathcal{Y}_n + \big(f(t_n,\mathcal{X}_n, \mathcal{Y}_n, \mathcal{Z}_n, \mathcal{I}_n) - f(t_n, \mathcal{X}_n, \hat{Y}_n, \hat{Z}_n, \hat{I}_n)\big)  \Delta t|^2\right\}\\
        & + \Delta t \bigg( \mathbb{E}\left\{|\hat{Z}_n - \mathcal{Z}_n|^2+| \hat{I}_n-\mathcal{I}_n|^2\right\}\bigg).\\
    \end{aligned}
\end{equation*}
Using the Young inequality, we have
$$
\begin{aligned}
    &\mathbb{E}|\hat{Z}_n - \mathcal{Z}_n|^2 \\
    & \leq (1+K) \mathbb{E}|Z_n - \mathcal{Z}_n|^2 + (1+\frac{1}{K}) \mathbb{E}|\hat{Z}_n - Z_n|^2\\
    & \leq (1+K) \mathbb{E}|Z_n - \mathcal{Z}_n|^2 + (1+\frac{1}{K}) \mathbb{E}| \frac{1}{\Delta t}\mathbb{E}_n \{(\mathcal{Y}_{n+1} - Y_{n+1})\Delta W_n\}  |^2\\
    & \leq (1+K) \mathbb{E}|Z_n - \mathcal{Z}_n|^2 + (1+\frac{1}{K})\frac{1}{\Delta t} \mathbb{E} \{(\mathcal{Y}_{n+1} - Y_{n+1})^2\},
\end{aligned}
$$
and 
$$
\begin{aligned}
    &\mathbb{E}|\hat{I}_n - \mathcal{I}_n|^2 \\
    & \leq (1+K) \mathbb{E}|I_n - \mathcal{I}_n|^2 + (1+\frac{1}{K}) \mathbb{E}| \frac{1}{\Delta t}\mathbb{E}_n \{(\mathcal{Y}_{n+1} - Y_{n+1})\int_E \tilde{\mu}((t_n,t_{n+1}],de)\}  |^2\\
    & \leq (1+K) \mathbb{E}|I_n - \mathcal{I}_n|^2 + (1+\frac{1}{K})(\int_E \lambda(de))^2 \mathbb{E} \{(\mathcal{Y}_{n+1} - Y_{n+1})^2\}. \\
\end{aligned}
$$
Then employing Young inequality, Cauchy-Schwarz inequality, Lipschitz condition on $f$ and previous estimates, we arrive at 
\begin{equation*}
    \begin{aligned}
        \mathcal{L}^{(1)}_n(\theta)\leq& (1+K\Delta t) \mathbb{E}\{|\hat{Y}_n-\mathcal{Y}_n|^2\} + K\Delta t \mathbb{E} \big\{ |\hat{Z}_n - \mathcal{Z}_n|^2+| \hat{I}_n-\mathcal{I}_n|^2\big\}, \\
        \leq& (1+K\Delta t) \mathbb{E}\{|\hat{Y}_n-\mathcal{Y}_n|^2\} + K\Delta t \mathbb{E} \big\{ |Z_n - \mathcal{Z}_n|^2+| I_n-\mathcal{I}_n|^2\big\}\\
        & + K (1+\Delta t)\mathbb{E} \big\{ |Y_{n+1} - \mathcal{Y}_{n+1}|^2\big\},\\
    \end{aligned}
\end{equation*}
and
\begin{equation*}
    \mathcal{L}^{(1)}_n(\theta)\geq (1 - \gamma \Delta t) \mathbb{E}\{|\hat{Y}_n-\mathcal{Y}_n|^2\} + \frac{\Delta t}{2} \mathbb{E} \big\{ |\hat{Z}_n - \mathcal{Z}_n|^2+| \hat{I}_n-\mathcal{I}_n|^2\big\}.
\end{equation*}
Recall the definition of the loss function \eqref{eqn:loss} to get
\begin{equation*}
    \begin{aligned}
        (N+1) &\mathcal{L}(\theta) = \sum_{n=0}^{N-1}\mathbb{E}\{|\mathcal{Y}_{n+1} - \mathcal{T}\left(t_n,\mathcal{X}_n,\mathcal{Y}_n,\mathcal{Z}_n\right)|^2\} + \mathbb{E}\{|\mathcal{Y}_N-g(\mathcal{X}_N)|^2\}\\
        = &  \sum_{n=0}^{N-1} \mathcal{L}^{(1)}_n(\theta)  + \mathbb{E}\{|\mathcal{Y}_N-g(\mathcal{X}_N)|^2\} + \sum_{n=0}^{N-1}\mathcal{L}_n^{(2)}\\
        \leq &  (1+K\Delta t) \sum_{n=0}^{N-1}\mathbb{E}\{|\hat{Y}_n-\mathcal{Y}_n|^2\} + K\Delta t \sum_{n=0}^{N-1}\mathbb{E} \big\{ |Z_n - \mathcal{Z}_n|^2+| I_n-\mathcal{I}_n|^2\big\}\\
        &\quad +K(1+\Delta t) \sum_{n=1}^{N}\mathbb{E} \big\{ |Y_{n} - \mathcal{Y}_{n}|^2\big\}+ \mathbb{E}\{|\mathcal{Y}_N-g(\mathcal{X}_N)|^2\}+ \sum_{n=0}^{N-1}\mathcal{L}_n^{(2)}.
    \end{aligned}
\end{equation*}
Similarly for the lower estimate
\begin{equation*}
    (N+1) \mathcal{L}(\theta) \geq (1 - \gamma \Delta t) \sum_{n=0}^{N-1}\mathbb{E}\{|\hat{Y}_n-\mathcal{Y}_n|^2\}+ \mathbb{E}\{|\mathcal{Y}_N-g(\mathcal{X}_N)|^2\}+\sum_{n=0}^{N-1} \mathcal{L}_n^{(2)}.\\
\end{equation*}

By the definition of $\theta^* = \arg\min_{\theta}\mathcal{L}(\theta)$, we obtain
\begin{equation*}
    \begin{aligned}
        &(1 - \gamma \Delta t) \sum_{n=0}^{N-1}\mathbb{E}\{|\hat{Y}_n-\mathcal{Y}_n(\theta^*)|^2\}+ \mathbb{E}\{|\mathcal{Y}_N(\theta^*)-g(\mathcal{X}_N)|^2\}\\
        & \quad \leq (1+K\Delta t) \sum_{n=0}^{N-1}\mathbb{E}\{|\hat{Y}_n-\mathcal{Y}_n|^2\} + K\Delta t \sum_{n=0}^{N-1}\mathbb{E} \big\{ |Z_n - \mathcal{Z}_n|^2+| I_n-\mathcal{I}_n|^2\big\}\\
        &\qquad\quad +K(1+\Delta t) \sum_{n=1}^{N}\mathbb{E} \big\{ |Y_{n} - \mathcal{Y}_{n}|^2\big\}+ \mathbb{E}\{|\mathcal{Y}_N-g(\mathcal{X}_N)|^2\},
    \end{aligned}
\end{equation*}
where $\mathcal{Y}_n(\theta^*)$ denotes the estimated solution $\mathcal{Y}_n$ with optimal parameter $\theta^*$ and $\mathcal{Y}_n$ denotes the estimated solution $\mathcal{Y}_N$ with arbitrary parameter $\theta$. Obviously by considering equation \eqref{eqn:estimate_y_n_hat_y_n_sum_condition}, we obtain
\begin{equation*}
    \begin{aligned}
        &\sum_{n=0}^{N-1}\mathbb{E}\{|\hat{Y}_n-\mathcal{Y}_n(\theta^*)|^2\}\\
        &\quad\quad \leq K \bigg\{ \sum_{n=0}^{N-1}\mathbb{E}\{|\hat{Y}_n-\mathcal{Y}_n|^2\} +(1+\Delta t)\sum_{n=1}^{N}\mathbb{E} \big\{ |Y_{n} - \mathcal{Y}_{n}|^2\big\} \\
        &\quad\quad\quad\qquad + \Delta t \sum_{n=0}^{N-1}\mathbb{E} \big\{ |Z_n - \mathcal{Z}_n|^2+| I_n-\mathcal{I}_n|^2\big\} + \mathbb{E}\{|\mathcal{Y}_N-g(\mathcal{X}_N)|^2\} \bigg\}\\
        &\quad\quad \leq K \bigg\{ (1+K)\sum_{n=0}^{N-1}\mathbb{E}\{|\hat{Y}_n-Y_n|^2\} +(2+\frac{1}{K}+\Delta t)\sum_{n=1}^{N}\mathbb{E} \big\{ |Y_{n} - \mathcal{Y}_{n}|^2\big\} \\
        &\quad\quad\quad\qquad + \Delta t \sum_{n=0}^{N-1}\mathbb{E} \big\{ |Z_n - \mathcal{Z}_n|^2+| I_n-\mathcal{I}_n|^2\big\} + \mathbb{E}\{|\mathcal{Y}_N-g(\mathcal{X}_N)|^2\} \bigg\}\\
        &\quad\quad \leq K \bigg\{\Delta t+ \epsilon^Y + \epsilon^Z+ \epsilon^I + (1+\Delta t)  \sum_{n=1}^{N}\mathbb{E}\{|Y_{n} - \mathcal{Y}_{n}|^2\} \\
        &\quad\quad\quad\qquad + \Delta t \sum_{n=0}^{N-1}\mathbb{E} \big\{ |Z_n - \mathcal{Z}_n|^2+| I_n-\mathcal{I}_n|^2\big\} + \mathbb{E}\{|\mathcal{Y}_N-g(\mathcal{X}_N)|^2\} \bigg\}.
    \end{aligned}
\end{equation*}
Substituting the above equation into \eqref{eqn:max_Y_n_cal_Y_n}, we obtain
\begin{equation}\label{eqn:final_estimate_with_NN}
    \begin{aligned}
        &\max_{n} \mathbb{E} \{|Y_n - \mathcal{Y}_n(\theta^*)|^2\} \leq K \bigg(\Delta t +  \epsilon^Y + \epsilon^Z+ \epsilon^I +\sum_{n=1}^{N}\mathbb{E} \big\{ |Y_{n} - \mathcal{Y}_{n}|^2\big\}\\
        &\qquad\quad + \Delta t \sum_{n=0}^{N-1}\mathbb{E} \big\{ |Z_n - \mathcal{Z}_n|^2+| I_n-\mathcal{I}_n|^2\big\} + \mathbb{E}\{|g(\mathcal{X}_N)-\mathcal{Y}_N|^2\}\bigg).
    \end{aligned}
\end{equation}

{\bf Step 4:}
Following \cite{hureDeepBackwardSchemes2020,castroDeepLearningSchemes2022}, for simplicity, we consider a single hidden layer neural network with $m$ neurons, where the activation function is three times continuously differentiable. Lemma \ref{lemma:UAT} implies
\begin{equation*}
    \lim_{m \to \infty} \epsilon^{\mathcal{Y}} = 0.
\end{equation*}
Let $\gamma_m$ denotes the total variation norm of the given neural networks. Since  $\mathcal{Z}_n$ is calculated through gradient of the given neural network, we have (see, e.g., \cite{hureDeepBackwardSchemes2020}),
\begin{equation*}
    \lim_{m,N \to \infty} \frac{\gamma_m^6}{N} = 0, \quad\text{and},\quad \max_n \mathbb{E} \{|Z_n - \mathcal{Z}_n|^2\} \leq K \big(\gamma_m^6 + \gamma_m^8 \Delta t^2 \big) \Delta t^2.
\end{equation*}
For the integral approximation error, noting \eqref{Taylor-integro} and the assumption \ref{assumption:A2}, we have
\begin{equation*}
    \mathbb{E} \big\{| I_n-\mathcal{I}_n|^2\big\} = \mathbb{E} \left\{ | \Delta t\int_{E} o\big(\beta (t_n,X_{n},e)\big) \lambda(de)|^2\right\} \leq K \left(\int_E \lambda(de)\right)^2 \Delta t^2.
\end{equation*}

Combining all the estimates for all terms, we obtain the desired result \eqref{eqn:Consistency of the FBSJNN} and complete the proof of the theorem.
\end{proof}

In a more simplified form, the estimate (\ref{eqn:Consistency of the FBSJNN}) can be rewritten as
\begin{align}\label{eq3.16}
    &\max_{n} \big(\mathbb{E} \{|Y_n - \mathcal{Y}_n(\theta^*)|^2\} + \mathbb{E} \{|Z_n - \mathcal{Z}_n(\theta^*)|^2\} + \mathbb{E} \big\{| I_n-\mathcal{I}_n(\theta^*)|^2\big\}\big)\nonumber\\
    \leq& K \big(\Delta t + (N+1)\epsilon^{\mathcal{Y}}\big).
\end{align}

We end this section with some remarks on our theoretical results and their proof.

Theorem \ref{thm:Consistency of the FBSJNN} and (\ref{eq3.16}) reveal that the error of the FBSJNN method is controlled by time discretization error and neural network approximation error. From \eqref{eq3.16}, in view of the Universal Approximation Theorem \ref{lemma:UAT} which implies $\lim_{m \to \infty} \epsilon^{\mathcal{Y}} = 0$, we can find that as the step-size of the time discretization Euler-Maruyama method for the FSDEs approaches to $0$, the neural network approximation $\mathcal{Y}_n(\theta^*)$, its gradient $\mathcal{Z}_n(\theta^*)$ and non-local integral $\mathcal{I}_n(\theta^*)$ converge to the exact solution $Y_n$, and the corresponding gradient $Z_n$ and integral $I_n$ of the BSDEs whenever the network width $m\to \infty$. This further implies the neural network approximation $\mathcal{Y}_n(\theta^*)=u(t_n,\mathcal X_n)$ converge to the solution $u(t_n,x)$ of PIDEs.

Although the inequality \eqref{eqn:max_Y_n_cal_Y_n} provides an estimate of the error between the exact solution $Y_n$ and the numerical approximation $\mathcal{Y}_n$, the neural network parameter $\theta$ can not be directly set to the optimal parameter defined in \eqref{eqn:optimal-para}. Specifically, the optimal parameter for the loss function \eqref{eqn:loss} may not correspond to the optimal parameter for the neural network that achieves the universal approximation as stated in Lemma \ref{lemma:UAT}. That is why we must continue refining our approach to obtain a more accurate estimate.

The results of the theorem reveal deeper implications regarding the parameters of the neural network in relation to the inequality \eqref{eqn:Consistency of the FBSJNN}. On the left-hand side, the parameter $\theta^*$ represents the optimal parameter obtained through a minimization process, which we refer to as the \textit{method error}. In contrast, the term on the right-hand side is encapsulated in the error $\epsilon^{\mathcal{Y}} = \inf_{m\in\mathbb{Z}^*} \max_n \mathbb{E}{|Y_n - \mathcal{Y}_n|^2}$, which derives from the Universal Approximation Theorem \ref{lemma:UAT}, and we refer to this as the \textit{universal error}. According to \eqref{eqn:Consistency of the FBSJNN}, the method error is bounded by the sum of the universal error and the regularities of the given problem. This relationship suggests that the proposed FBSJNN method, particularly the design of the loss function \eqref{eqn:loss}, is well-suited to guide the optimization algorithm toward achieving the best possible approximation.

Generally speaking, we address a gap in understanding whether a loss function can genuinely guide optimization algorithms toward an optimal approximation. Current work tends to focus on proposing new loss functions, often overlooking an assessment of their soundness. Our method approaches this issue from a stochastic analysis perspective, providing a logical framework for evaluating this relationship.

\section{Numerical Examples}
\label{section:Numerical Examples}
In this section, we provide numerical examples to validate the effectiveness of the FBSJNN method proposed in this paper. Some numerical examples are adapted from \cite{luTemporalDifferenceLearning2023, raissiForwardBackwardStochastic2018}. The code for this paper is available at \url{https://github.com/yezaijun/FBSJNN}.

\subsection{One-dimensional pure jump equation}
\label{subsection:purejump}
Consider the following pure jump process \cite{luTemporalDifferenceLearning2023}
\begin{equation*}
    \left\{
    \begin{aligned}
        \frac{\partial u}{\partial t} + \int_{\mathbb{R}}\big(u(t,xe^z)-u(t,x)-x(e^z-1)\frac{\partial u}{\partial x}\big)\nu(de)=0, \\
        u(T,x)=x,
    \end{aligned}
    \right.
\end{equation*}
where $\nu(de) = \lambda \phi(e)de $, and $\phi(e)$ is the probability density function of a normal distribution with mean $\mu_\phi$ and variance $\sigma_\phi^2$. Here, constant $\lambda$ is the intensity of poisson jumps. The true solution to this problem is $u(t,x)=x$. The corresponding discrete FBSDEJ is given by
\begin{equation*}
    \left\{
    \begin{aligned}
        X_{n+1} = & X_n +  \sum_{i=1}^{\mu_n} X_n(e^{z_i}-1) - \lambda X_n (e^{\mu_\phi+\frac{1}{2}\sigma_\phi^2}-1) \Delta t,      \\
        Y_{n+1} = & Y_{n} + \sum_{i=1}^{\mu_n} \hat{u}_n(z_i) - \lambda X_n (e^{\mu_\phi+\frac{1}{2}\sigma_\phi^2}-1)  \frac{\partial Y_{n}}{\partial x} \Delta t.
    \end{aligned}
    \right.
\end{equation*}
Taking $T=1, X_0 = 1, \lambda = 0.3, \mu_\phi = 0.4, \sigma_\phi = 0.25$, and dividing time into equal intervals with $N = 50$, we designed a neural network with 2 hidden layers, each with 16 neurons, and ReLU activation function. In each iteration of the neural network, we used the Adam optimizer, and simulated 1000 trajectories of $\{\mathcal{X}_n\}$ using the Monte-Carlo method as input data. After $5000$ training iterations, the numerical results are given in Fig. \ref{1dimpure_loss_MRE} and Table \ref{tab:PureJump}.

\begin{figure}[htbp]
    \centering
    \begin{tabular}{cc}
        \begin{subfigure}{0.45\textwidth}
            \centering
            \includegraphics[width=\textwidth]{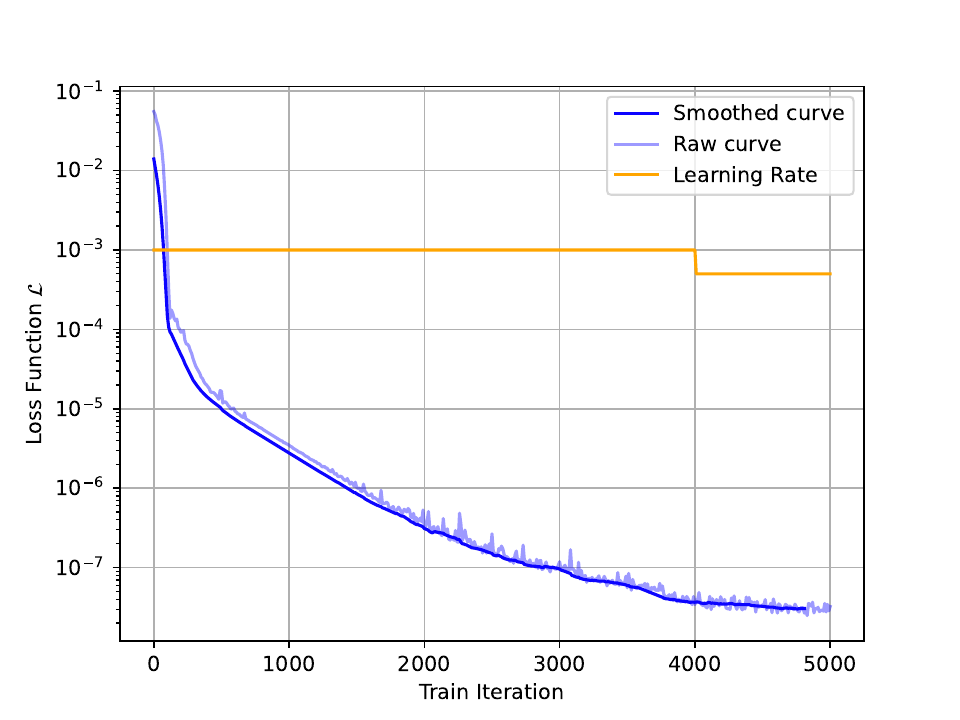}
            \caption{Loss function}
            \label{fig:1dpPIDE_loss}
        \end{subfigure}   &
        \begin{subfigure}{0.45\textwidth}
            \centering
            \includegraphics[width=\textwidth]{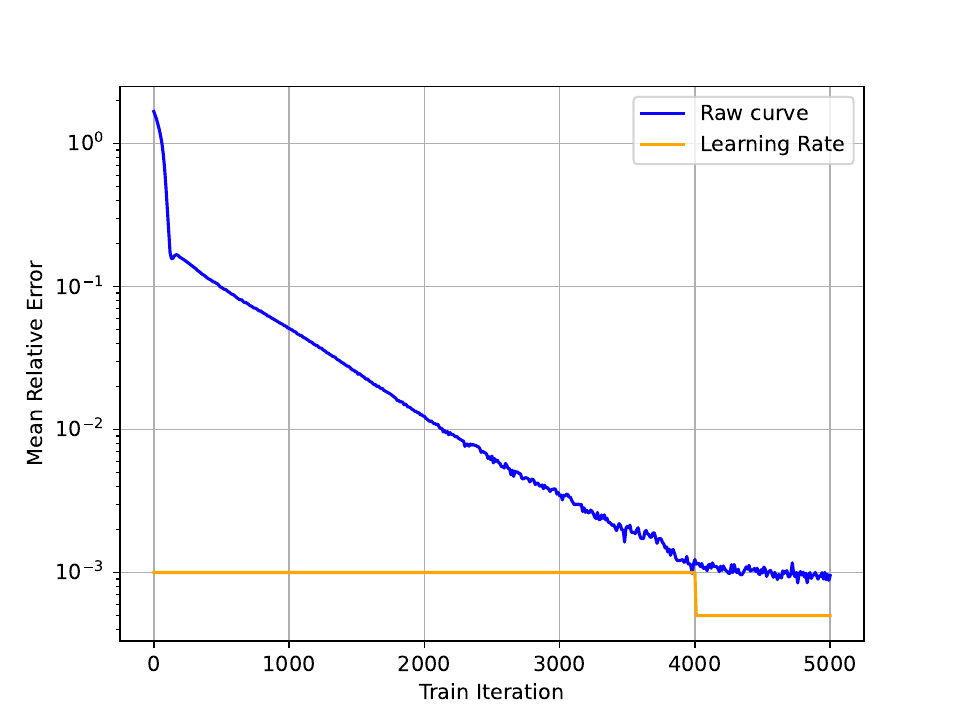}
            \caption{Mean relative error}
            \label{fig:1dpPIDE_MRE}
        \end{subfigure}    \\
        \begin{subfigure}{0.45\textwidth}
            \centering
            \includegraphics[width=\textwidth]{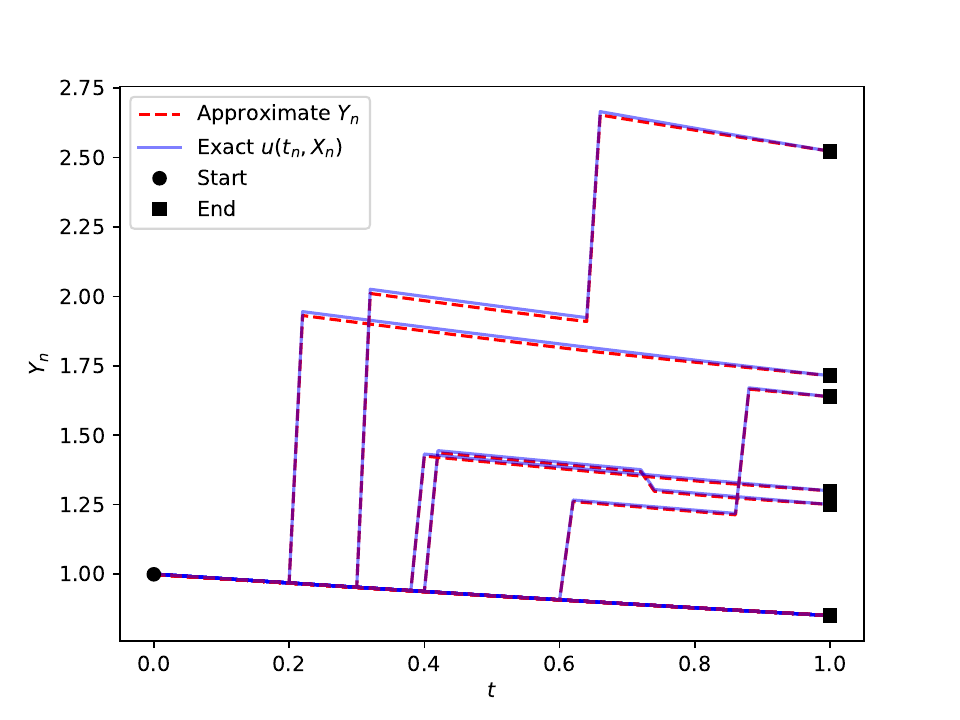}
            \caption{Sample path fitting}
            \label{fig:1dpPIDE_sample}
        \end{subfigure} &
        \begin{subfigure}{0.45\textwidth}
            \centering
            \includegraphics[width=\textwidth]{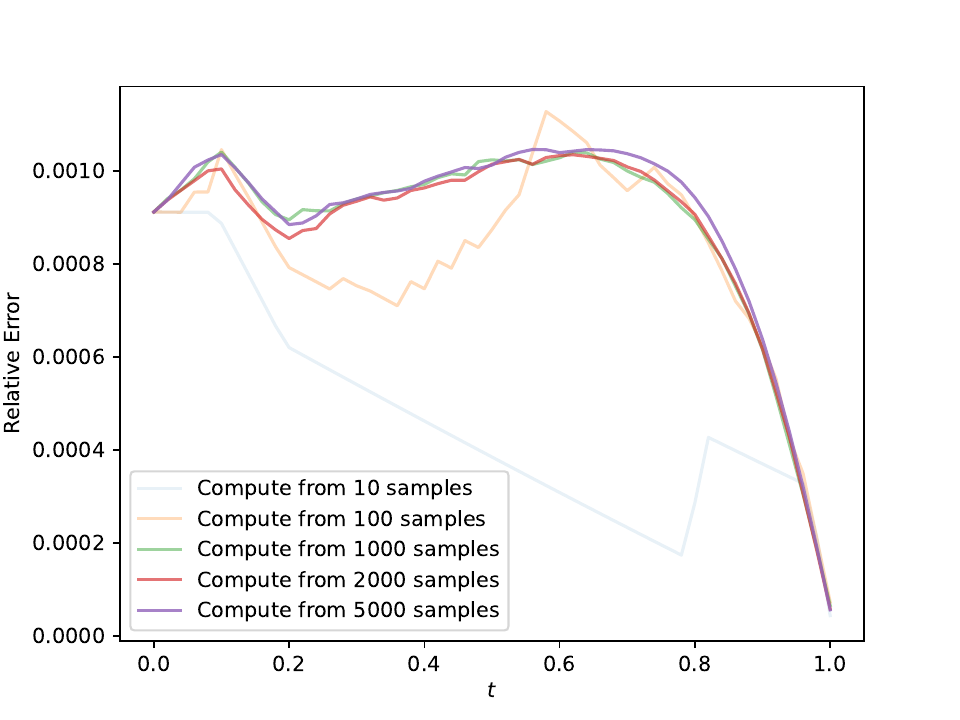}
            \caption{Mean relative error by time}
            \label{fig:1dpPIDE_erbyt}
        \end{subfigure}  \\
    \end{tabular}
    \caption{Solution of the one-dimensional pure jump equation.}
    \label{1dimpure_loss_MRE}
\end{figure}

\begin{table}[h]
    \centering
    \caption{Numerical results of neural network approximation for one-dimensional pure jump equation.}
    \label{tab:PureJump}
    \begin{tabular}{cccc}
        \toprule
        Iteration & Loss Function & Relative Error & Learning Rate \\
        \midrule
        1000      & 3.46e-06      & 5.09\%         & 0.0010        \\
        2000      & 3.17e-07      & 1.24\%         & 0.0010        \\
        3000      & 1.19e-07      & 0.34\%         & 0.0010        \\
        4000      & 3.42e-08      & 0.12\%         & 0.0010        \\
        5000      & 3.23e-08      & 0.10\%         & 0.0005        \\ \bottomrule
    \end{tabular}
\end{table}

Figure \ref{1dimpure_loss_MRE} shows that as the neural network iterations increase, the loss function value continuously decreases and gradually levels off. The average relative error also shows a similar trend. After simulating the trajectories of $\{\mathcal{X}_n\}$, we found that the network fitting results are good. By calculating the relative error at different time points, we found that the error over the entire time interval can be controlled within 0.1\%, and the relative error at the terminal condition is even closer to 0. Table \ref{tab:PureJump} reports the specific values during the iteration process. It can be seen that after 5000 iterations, the relative error of the final numerical solution can be controlled to 0.1\%, indicating that the method in this paper achieves rapid convergence and high accuracy for the one-dimensional pure jump problem, and has better performance compared to Lu et al. (2023) \cite{luTemporalDifferenceLearning2023}.

\subsection{One-dimensional partial integro-differential equation}

Using the notation from the pure jump equation, we consider the following more general PIDE (with diffusion and convection terms):
\begin{equation*}
    \label{eqn:numercial_oneDimPIDE}
    \left\{
    \begin{aligned}
        \frac{\partial u}{\partial t}(t,x)+ & \frac{1}{2}\tau^2 \frac{\partial u^2}{\partial^2 x}(t,x) + \epsilon x \frac{\partial u}{\partial x}(t,x) \\
        +                                   & \int_{\mathbb{R}}(u(t,xe^z)-u(t,x)-x(e^z-1)\frac{\partial u}{\partial x})\nu(de)=\epsilon x ,              \\
        u(T,x)=                             & x.
    \end{aligned}
    \right.
\end{equation*}
The true solution to this problem is $u(t,x)=x$. The corresponding discrete FBSDEJ format is
\begin{equation*}
    \left\{
    \begin{aligned}
        X_{n+1} = X_n +   & \epsilon x \Delta t + \tau \Delta W_n
        +\sum_{i=1}^{\mu_n} X_n(e^{z_i}-1) - \lambda X_n (e^{\mu_\phi+\frac{1}{2}\sigma_\phi^2}-1) \Delta t, \\
        Y_{n+1} = Y_{n} + & \epsilon x \Delta t + Z_n \Delta W_n
        + \sum_{i=1}^{\mu_n} \hat{u}_n(z_i) - \lambda X_n (e^{\mu_\phi+\frac{1}{2}\sigma_\phi^2}-1)  \frac{\partial Y_{n}}{\partial x} \Delta t.
    \end{aligned}
    \right.
\end{equation*}

\begin{table}[h]
    \centering
    \caption{Numerical results of neural network approximation for one-dimensional PIDE.}
    \label{table:1dPIDE_iteration_data}
    \begin{tabular}{cccc}
        \toprule
        Iteration Steps & Loss Function & Relative Error & Learning Rate \\
        \midrule
        2000            & 2.45e-06      & 1.37\%         & 1e-3          \\
        4000            & 3.41e-07      & 0.30\%         & 1e-3          \\
        6000            & 2.31e-07      & 0.25\%         & 1e-6          \\
        8000            & 2.94e-07      & 2.58\%         & 1e-6          \\
        10000           & 3.76e-07      & 0.39\%         & 1e-8          \\
        \bottomrule
    \end{tabular}
\end{table}

We take $T=1, X_0 = 1, \lambda = 0.3, \tau = 0.4, \epsilon=0.25, \mu_\phi = 0.4, \sigma_\phi = 0.25$, divide time equally into $N = 50$ as in Section \ref{subsection:purejump}, and keep the neural network model and training data settings unchanged. In this example, 10000 iterations were performed.

\begin{figure}[h]
    \centering
    \begin{tabular}{cc}
        \begin{subfigure}{0.45\textwidth}
            \centering
            \includegraphics[width=\textwidth]{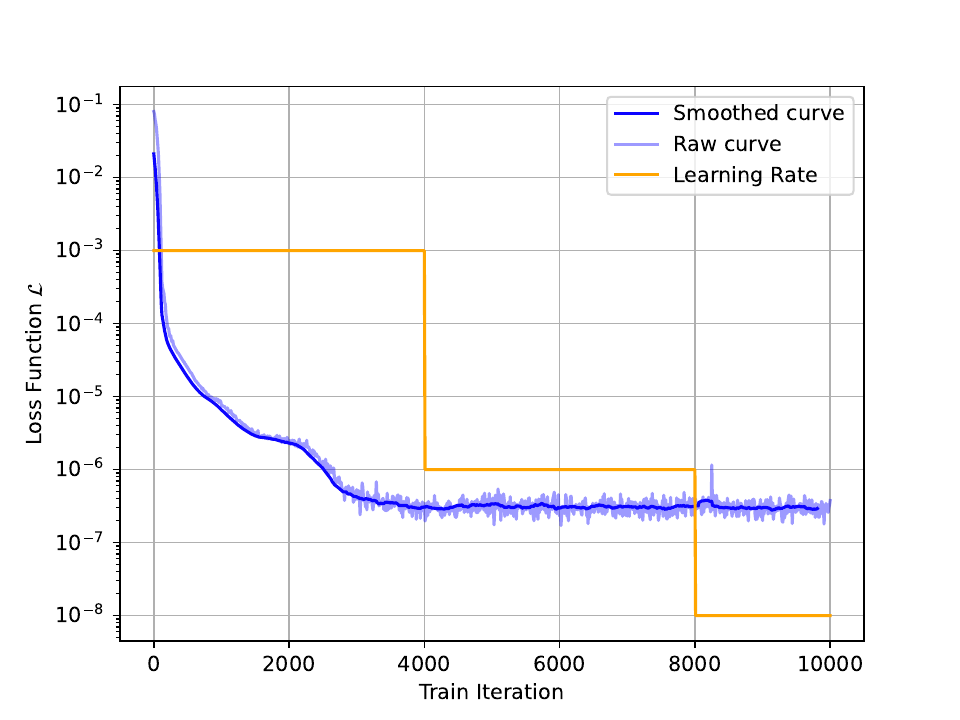}
            \caption{Loss function}
            \label{fig:1dPIDE_loss}
        \end{subfigure}   &
        \begin{subfigure}{0.45\textwidth}
            \centering
            \includegraphics[width=\textwidth]{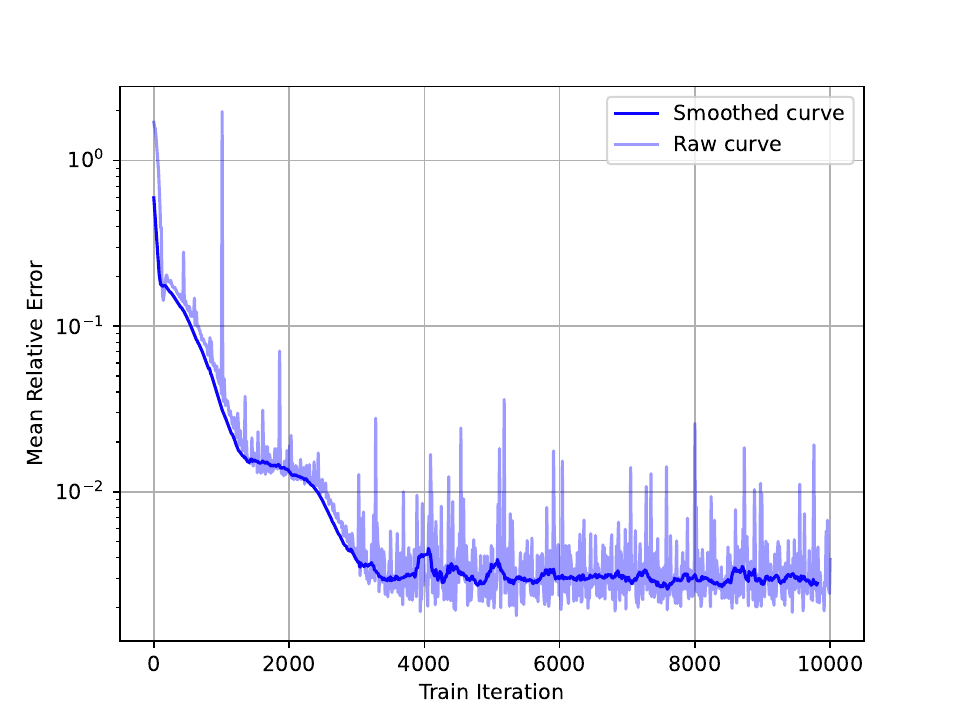}
            \caption{Relative error}
            \label{fig:1dPIDE_MRE}
        \end{subfigure}    \\
        \begin{subfigure}{0.45\textwidth}
            \centering
            \includegraphics[width=\textwidth]{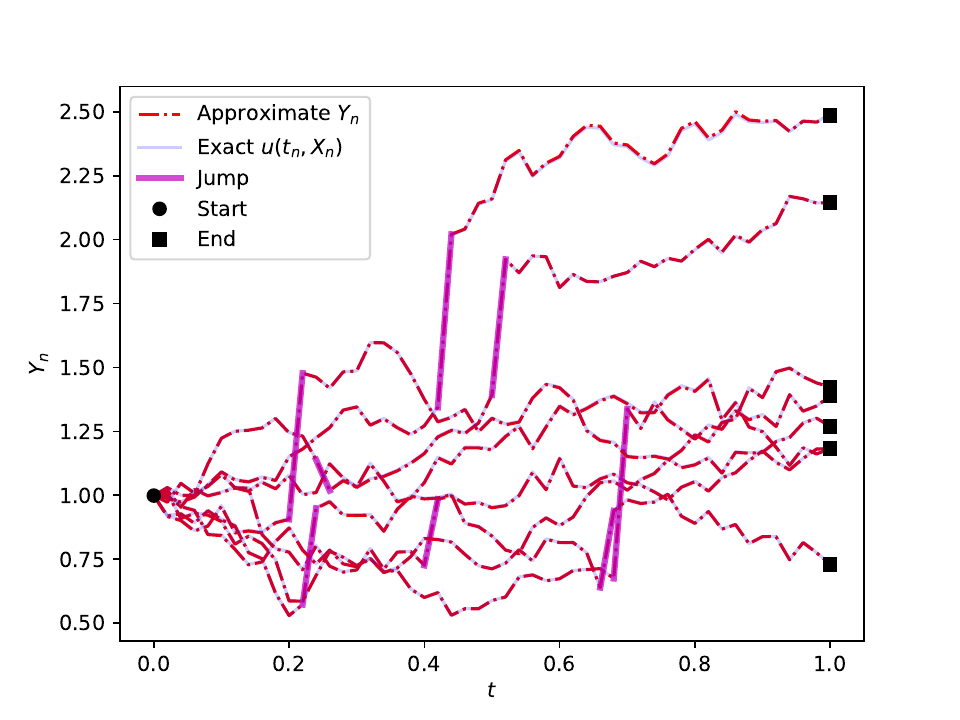}
            \caption{Sample path fitting}
            \label{fig:1dPIDE_sample}
        \end{subfigure} &
        \begin{subfigure}{0.45\textwidth}
            \centering
            \includegraphics[width=\textwidth]{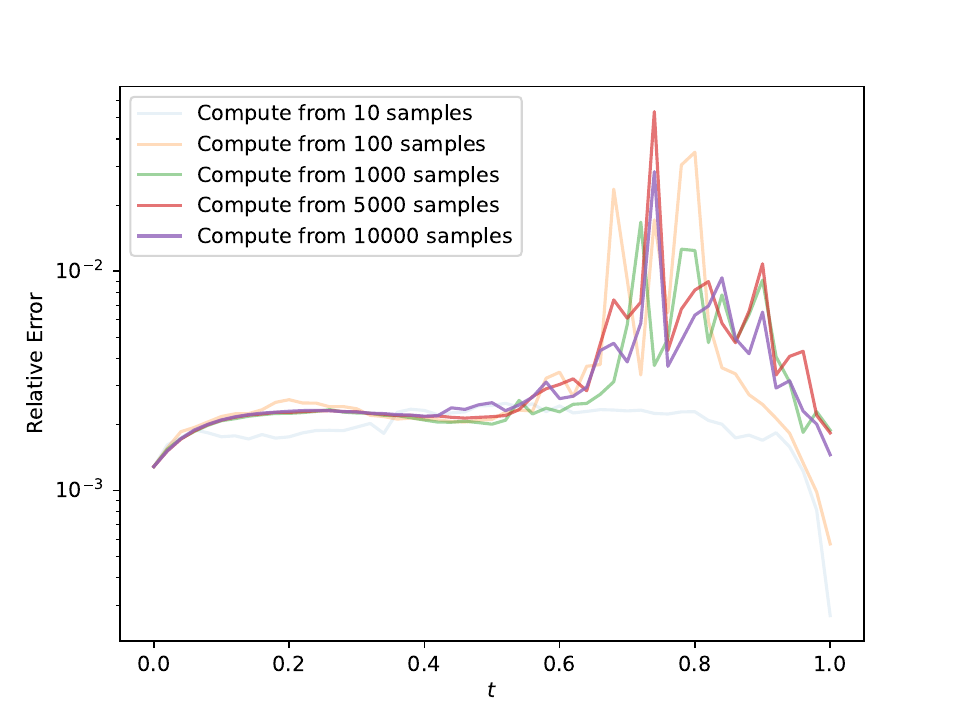}
            \caption{Relative error at time nodes}
            \label{fig:1dPIDE_erbyt}
        \end{subfigure}  \\
    \end{tabular}
    \caption{Solution of the one-dimensional partial integro-differential equation.}
    \label{fig:1dimPIDE_loss_MRE}
\end{figure}

Based on the training results in Table \ref{table:1dPIDE_iteration_data} and Fig. \ref{fig:1dimPIDE_loss_MRE}, after 4000 iterations, the loss function fluctuates and no longer decreases. This may be due to the small scale of the neural network, which has reached an optimal point and is difficult to further optimize. For the relative error, we also observe that it reaches the magnitude of 0.1\% at 4000 steps, within an acceptable range. To avoid overfitting and maintain the generalization performance of the model, we keep the current network structure and use the network weights at 4000 steps as the final numerical solution. From Figure \ref{fig:1dPIDE_erbyt}, we can see that the current network fits well at the endpoints $t=0$ and $t=1$, while there is a larger relative error for $t\in[0.6,0.9]$. This is because the endpoint $t=0$ is the starting point of all trajectories, with abundant samples favorable for model fitting; while the endpoint $t=1$ has a definite terminal value condition $g(\cdot)$, also aiding in model fitting. In contrast, the trajectories in the middle time range are more scattered, and the loss is controlled by the recursive formula of the BSDE, which easily leads to error propagation, making the model fitting more challenging.

\begin{figure}[h]
    \centering
    \includegraphics[width=0.8\textwidth]{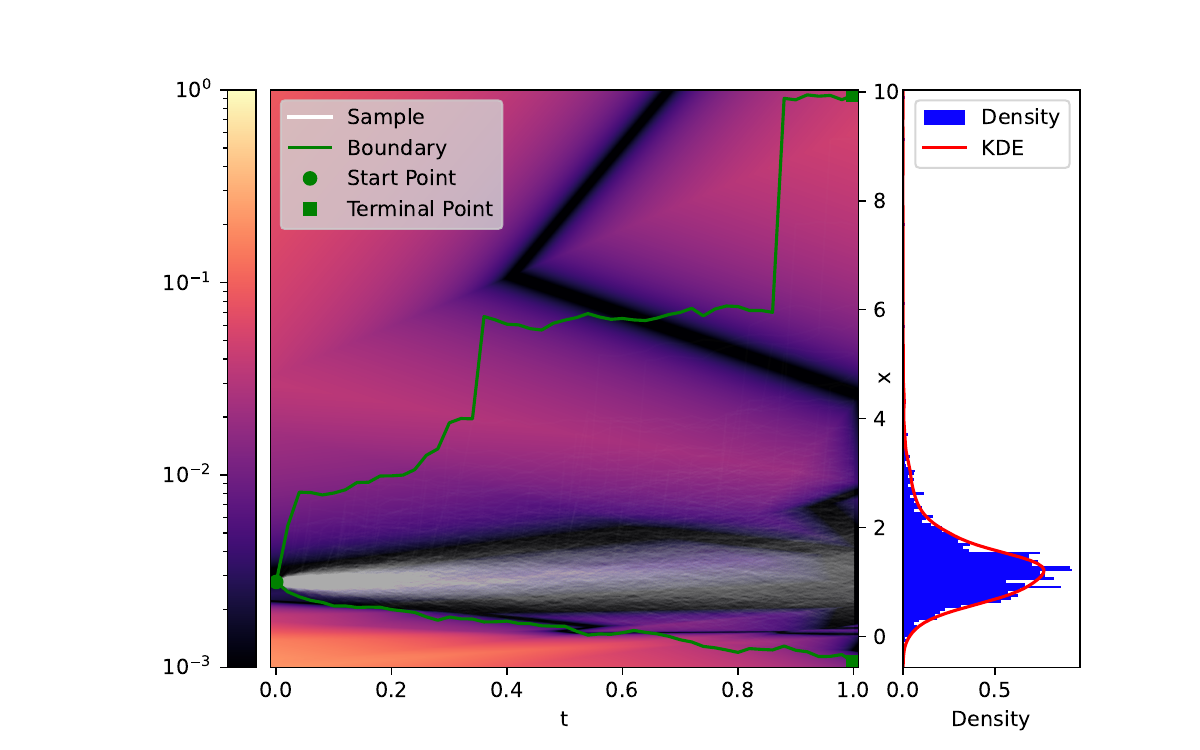}
    \caption{Heat map of absolute errors in spatio-temporal region.}
    \label{fig:1dPIDE_hot}
\end{figure}

To further examine the effectiveness of the numerical algorithm in the spatio-temporal region, we consider 2000 trajectory paths and calculate the absolute errors at all points in this interval, as shown in Figure \ref{fig:1dPIDE_hot}. The green lines represent the upper and lower bounds of 2000 trajectories, the white transparent lines represent the sample paths of the FSDE, and each colored block represents the absolute error of the numerical solution at that point. The histogram on the right shows the distribution of trajectory endpoints, where the red line is the kernel density estimation (KDE) of the endpoint distribution. This figure illustrates that the endpoint of FSDE trajectories mainly concentrates between $(0,2)$, where the absolute error of the numerical solution can be less than 1\%; however, in other areas, the absolute errors of the neural network fitting values are larger, which is related to the fewer samples reaching these positions.


\subsection{High-dimensional partial integro-differential equation}

To further verify the effectiveness of the numerical methods proposed in this paper for high-dimensional problems, we consider the following equation
\begin{equation*}
    \left\{
    \begin{aligned}
        \frac{\partial u}{\partial t}(t,x)+ & \frac{1}{2}\mathrm{Tr}(\tau^2 \nabla_x^2 u(t,x)) +\big\langle \frac{1}{2}\epsilon x ,\nabla_x u(t,x) \big\rangle                                                  \\
        +                                   & \int_{\mathbb{R}}(u(t,x+e)-u(t,x)-\big\langle e,\nabla_x u(t,x)\big\rangle)\nu(de)= \lambda (\mu_{\phi}^2+\sigma_{\phi}^2) + \tau^2 + \frac{\epsilon}{d} ||x||^2, \\
        u(T,x)=                             & \frac{1}{d} ||x||^2.
    \end{aligned}
    \right.
\end{equation*}
The true solution to this problem is $u(t,x)=\frac{1}{d} ||x||^2$. The corresponding discrete FBSDEJ is:

\begin{equation*}
    \left\{
    \begin{aligned}
        X_{n+1} = X_n +   & \frac{1}{2}\epsilon X_n \Delta t + \tau \Delta W_n
        +\sum_{i=1}^{\mu_n} z_i - \lambda \mu_\phi \Delta t,\\
        Y_{n+1} = Y_{n} + & \big(\lambda (\mu_{\phi}^2+\sigma_{\phi}^2) + \tau^2 + \frac{\epsilon}{d} ||x||^2\big) \Delta t + Z_n^\top \Delta W_n\\
        &+ \sum_{i=1}^{\mu_n} \hat{u}_n(z_i) - \lambda \mu\phi  \nabla_x Y_{n}^\top \mathbf{1} \Delta t.
    \end{aligned}
    \right.
\end{equation*}

Taking $d = 100,~T=1,~ X_0 = \mathbf{1}=(1,1,\cdot,1)^T, ~\lambda = 0.3, ~\tau = 0.1, ~\mu_\phi = 0.01, \sigma_\phi = 0.1$, and equidistant time division $N = 50$, we use an Adam optimizer with a neural network containing 2 hidden layers, each with 256 neurons, and a Leaky ReLU function as the active function. After $30,000$ iterations, the final result is shown in Fig. \ref{fig:highdimPIDE_loss_MRE}.

\begin{figure}[h]
    \centering
    \begin{tabular}{cc}
        \begin{subfigure}{0.45\textwidth}
            \centering
            \includegraphics[width=\textwidth]{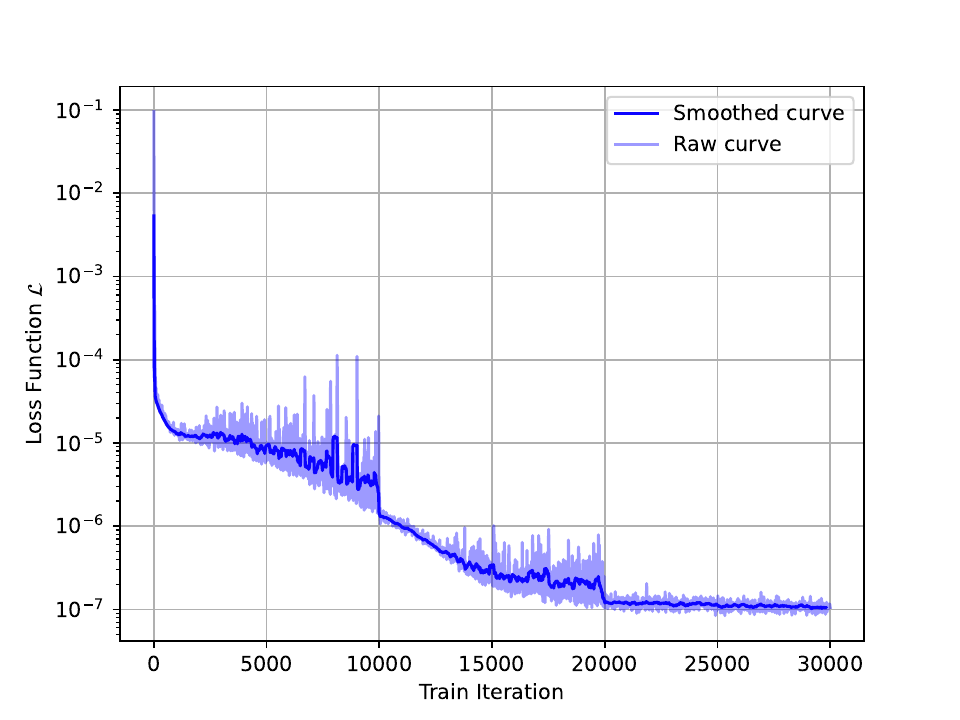}
            \caption{Loss function}
            \label{fig:hdPIDE_loss}
        \end{subfigure}   &
        \begin{subfigure}{0.45\textwidth}
            \centering
            \includegraphics[width=\textwidth]{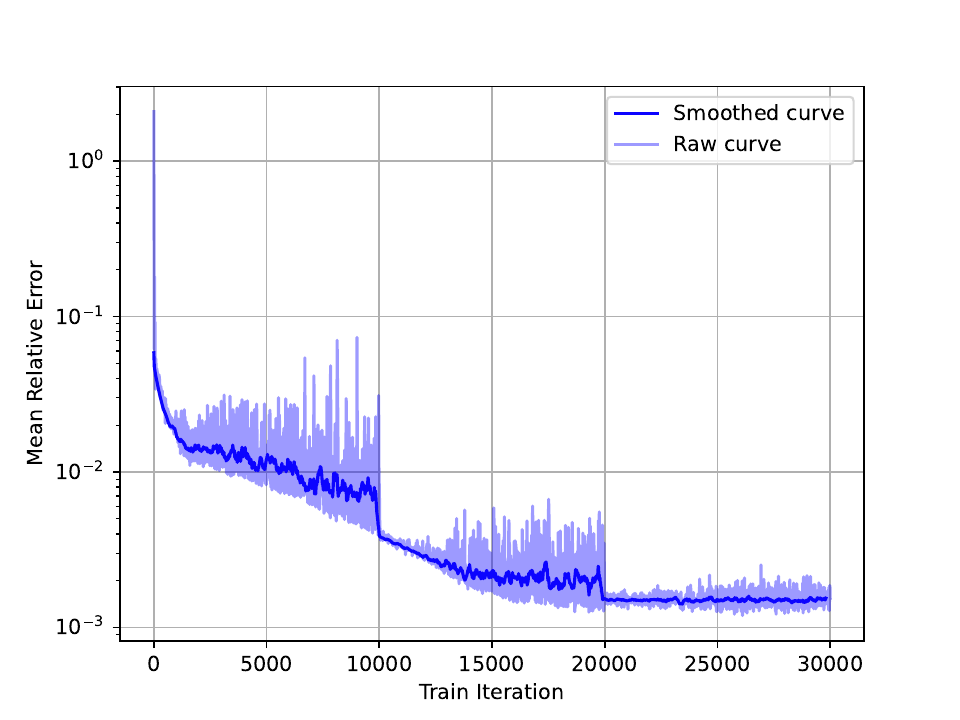}
            \caption{Relative error}
            \label{fig:hdPIDE_MRE}
        \end{subfigure}    \\
        \begin{subfigure}{0.45\textwidth}
            \centering
            \includegraphics[width=\textwidth]{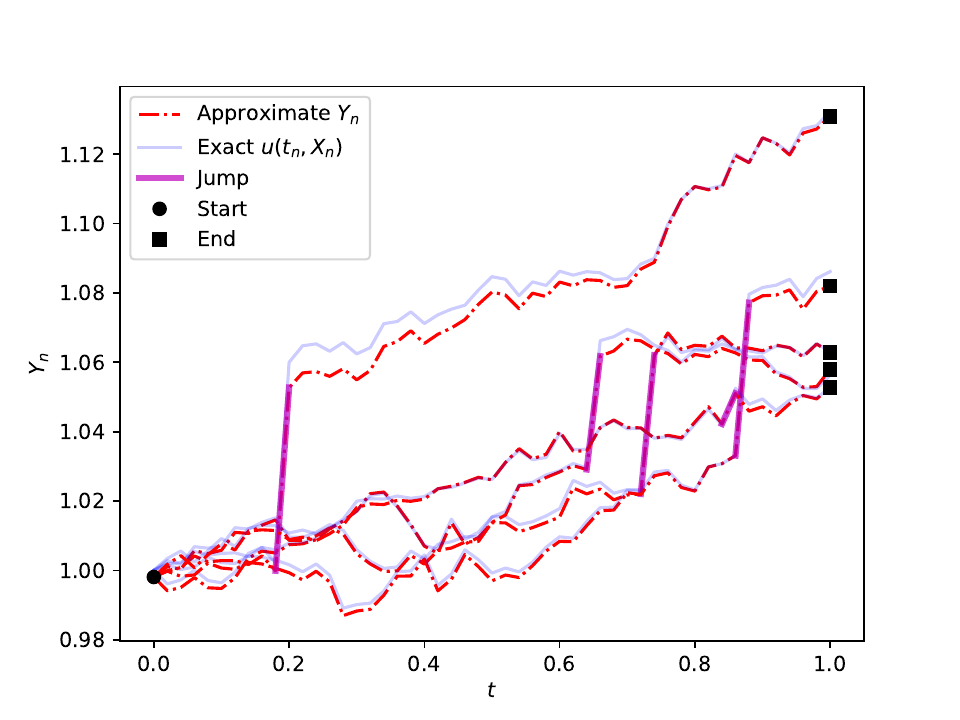}
            \caption{Sample path fitting}
            \label{fig:hdPIDE_sample}
        \end{subfigure} &
        \begin{subfigure}{0.45\textwidth}
            \centering
            \includegraphics[width=\textwidth]{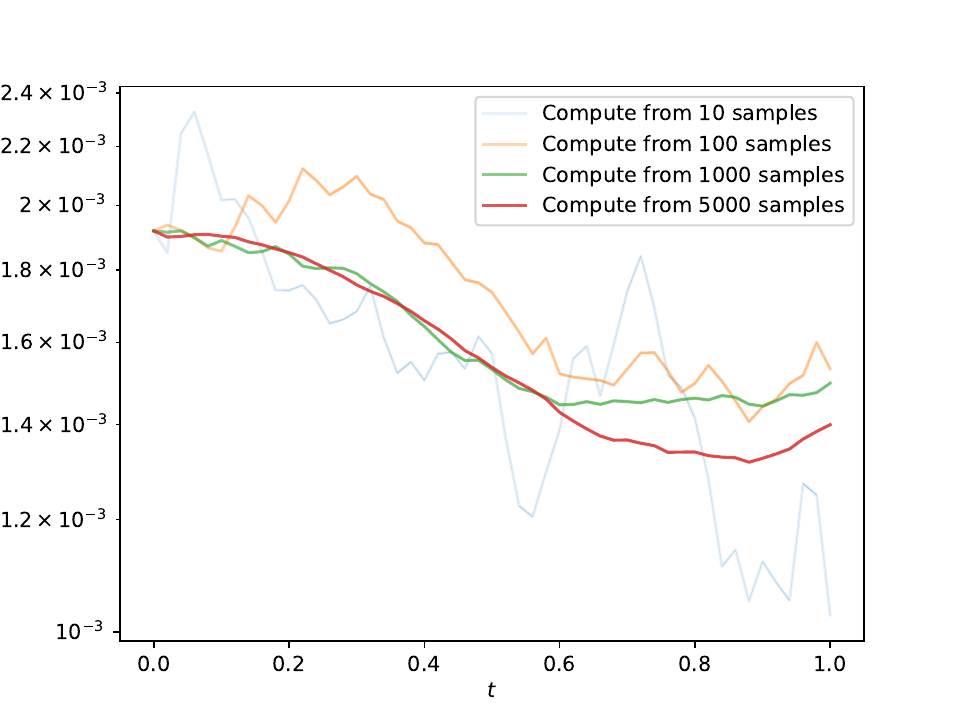}
            \caption{Relative error at time nodes}
            \label{fig:hdPIDE_erbyt}
        \end{subfigure}  \\
    \end{tabular}
    \caption{Solution of the high-dimensional partial integro-differential equations.}
    \label{fig:highdimPIDE_loss_MRE}
\end{figure}

The results in Fig. \ref{fig:highdimPIDE_loss_MRE} demonstrate the effectiveness of the FBSJNN method. The loss function and mean relative error decrease steadily and stabilize at very low values after approximately $25,000$ training iterations, indicating the method's convergence and high accuracy. In Fig. \ref{fig:highdimPIDE_loss_MRE} (C), the approximate solution \( Y_t \) aligns closely with the exact solution \( u(t, x) \), effectively capturing the jump dynamics and boundary behavior. Furthermore, Fig. \ref{fig:highdimPIDE_loss_MRE} (D) shows that the relative error decreases as the sample size increases, with significant improvements observed as the sample size grows from $10$ to $5,000$, highlighting the method's robustness and scalability. Overall, FBSJNN achieves accurate solutions, demonstrating its capability to handle the complexity of high-dimensional PIDEs.

\subsection{High-dimensional Black-Scholes-Barenblatt equations}
In final example, we consider high-dimensional Black-Scholes-Barenblatt equation with jumps, adapted from Raissi \cite{raissiForwardBackwardStochastic2018}. The corresponding FBSDEJ is discretized by
\begin{equation}
    \label{eqn:highdimBSB}
    \left\{
    \begin{aligned}
        X_{n+1} = X_n +   & r X_n \Delta t + \tau \text{diag}(X_n) \Delta W_n
        +\sum_{i=1}^{\mu_n} z_i - \lambda \mu_\phi \Delta t,\\
        Y_{n+1} = Y_{n} + & \big(r Y_n + \lambda e^{(r+\tau^2)(T-t)}(\mu_\phi^2+\sigma_\phi^2)) \Delta t + Z_n^\top \Delta W_n \\
                          & \qquad\qquad\qquad\qquad\qquad\qquad+ \sum_{i=1}^{\mu_n} \hat{u}_n(z_i) - \lambda \mu_\phi  \nabla_x Y_{n}^\top\mathbf{1} \Delta t.
    \end{aligned}
    \right.
\end{equation}
The terminal condition is $u(T,x) = \frac{1}{d}||x||^2$, and its true solution is $u(t,x) = \frac{1}{d} e^{(r+\tau^2)(T-t)} ||x||^2$. We still take $d = 100, ~T=1, ~X_0 = \mathbf{1}, ~\lambda = 0.3,~r=0.05,~ \tau = 0.4,~ \mu_\phi = 0.02, ~\sigma_\phi = 0.01$, and divide time equally into $N = 50$ intervals. We use the Adam optimizer for a neural network with $5$ hidden layers, each with $128$ neurons, and a Leaky ReLU activation function. After $5000$ iterations, we obtain the final results which are presented in  Table \ref{tab:highBSB}. From Table \ref{tab:highBSB}, we observe that after 5000 iterations, a numerical solution with a relative error of 0.13\% at the initial value is obtained.

\begin{table}[h]
    \centering
    \caption{Numerical results of neural network approximation for high-dimensional B-S-B equations}
    \label{tab:highBSB}
    \begin{tabular}{ccccc}
        \toprule
        Iteration & Loss Function & Relative Error & Relative Error at $t_0$ & Learning Rate \\
        \midrule
        0         & 5.57e-03      & 37.91\%        & 37.62\%                 & 1.00e-03      \\
        1000      & 2.00e-05      & 1.14\%         & 1.00\%                  & 1.00e-03      \\
        2000      & 2.00e-05      & 1.32\%         & 1.30\%                  & 1.00e-03      \\
        3000      & 1.30e-05      & 0.80\%         & 0.02\%                  & 1.00e-04      \\
        4000      & 1.20e-05      & 0.82\%         & 0.19\%                  & 1.00e-04      \\
        5000      & 1.30e-05      & 0.77\%         & 0.13\%                  & 1.00e-05      \\
        \bottomrule
    \end{tabular}
\end{table}

\subsection{Time discretization error test} From Theorem \ref{thm:Consistency of the FBSJNN} we know that the final approximation error of the FBSJNN method also depends on the time step-size $\Delta t$. In this subsection, we will use the numerical results of the two-dimensional B-S-B problems \eqref{eqn:highdimBSB} to verify this.

We vary the time step-size to investigate its impact on the numerical solution. In deep learning practice, various factors can influence the final learning performance of a neural network. To ensure a fair comparison, we fix the neural network architecture to consist of two hidden layers, each with $32$ neurons, and adopt the Tanh activation function. Each model is trained for $2000$ iterations using the Adam optimizer with a cosine annealing learning rate scheduler, which decreases from $10^{-2}$ to $10^{-5}$. The sample size is fixed at $1000$. For each time step-size, we conduct $10$ independent trainings, and take the minimum $3$ training results to calculate the mean and variance of the metrics. The results are illustrated in Fig. \ref{fig:different_N}. The max square error is defined like Theorem \ref{thm:Consistency of the FBSJNN} as follows:
$$\max_{n} \mathbb{E} \{|Y_n - \mathcal{Y}_n(\theta^*)|^2\}.$$

\begin{figure}[h]
    \centering
    \begin{tabular}{ccc}
        \begin{subfigure}{0.3\textwidth}
            \centering
            \includegraphics[width=\textwidth,trim=1.2cm 1.1cm 2cm 1.3cm, clip]{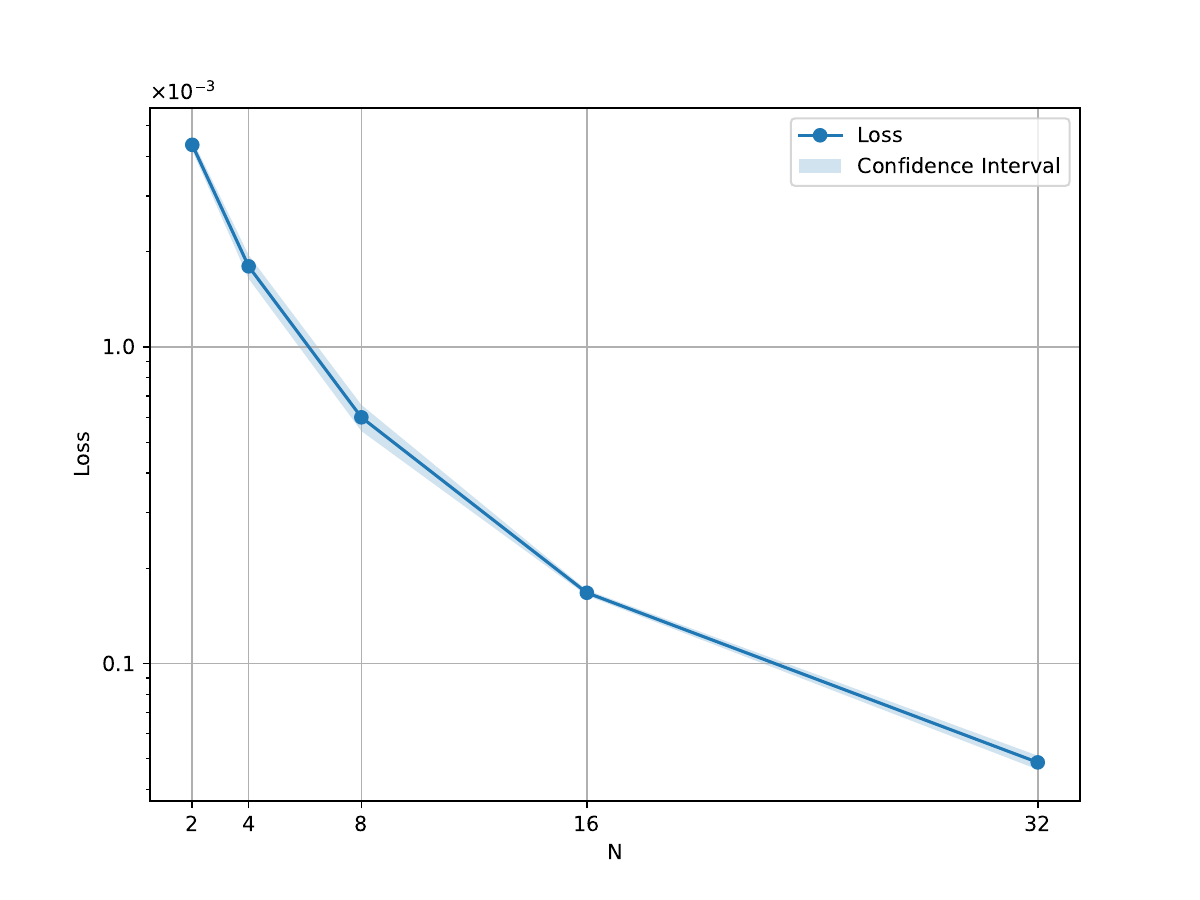}
            \caption{Loss function}
        \end{subfigure}   &
        \begin{subfigure}{0.3\textwidth}
            \centering
            \includegraphics[width=\textwidth,trim=1.2cm 1.1cm 2cm 1.3cm, clip]{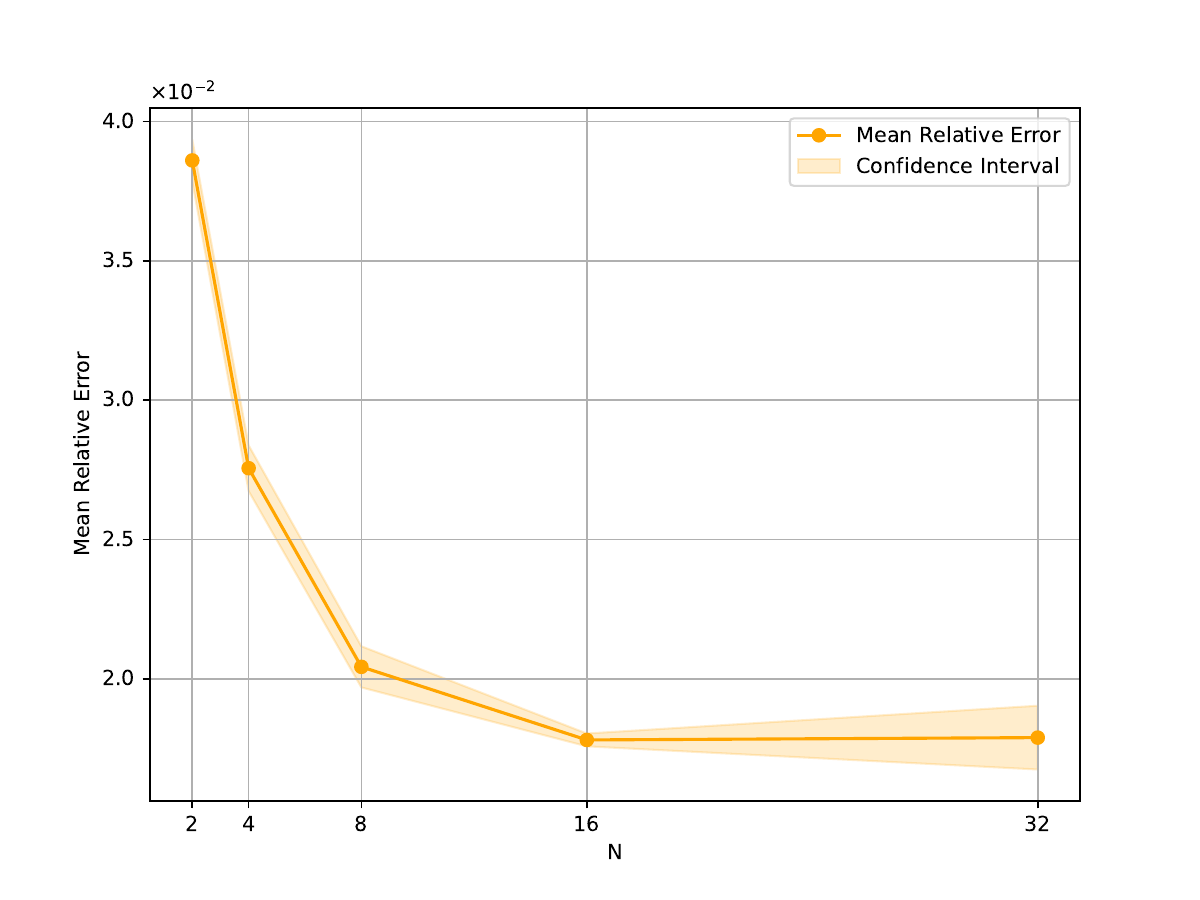}
            \caption{Mean relative error}
        \end{subfigure}   &
        \begin{subfigure}{0.3\textwidth}
            \centering
            \includegraphics[width=\textwidth,trim=1.2cm 1.1cm 2cm 1.3cm, clip]{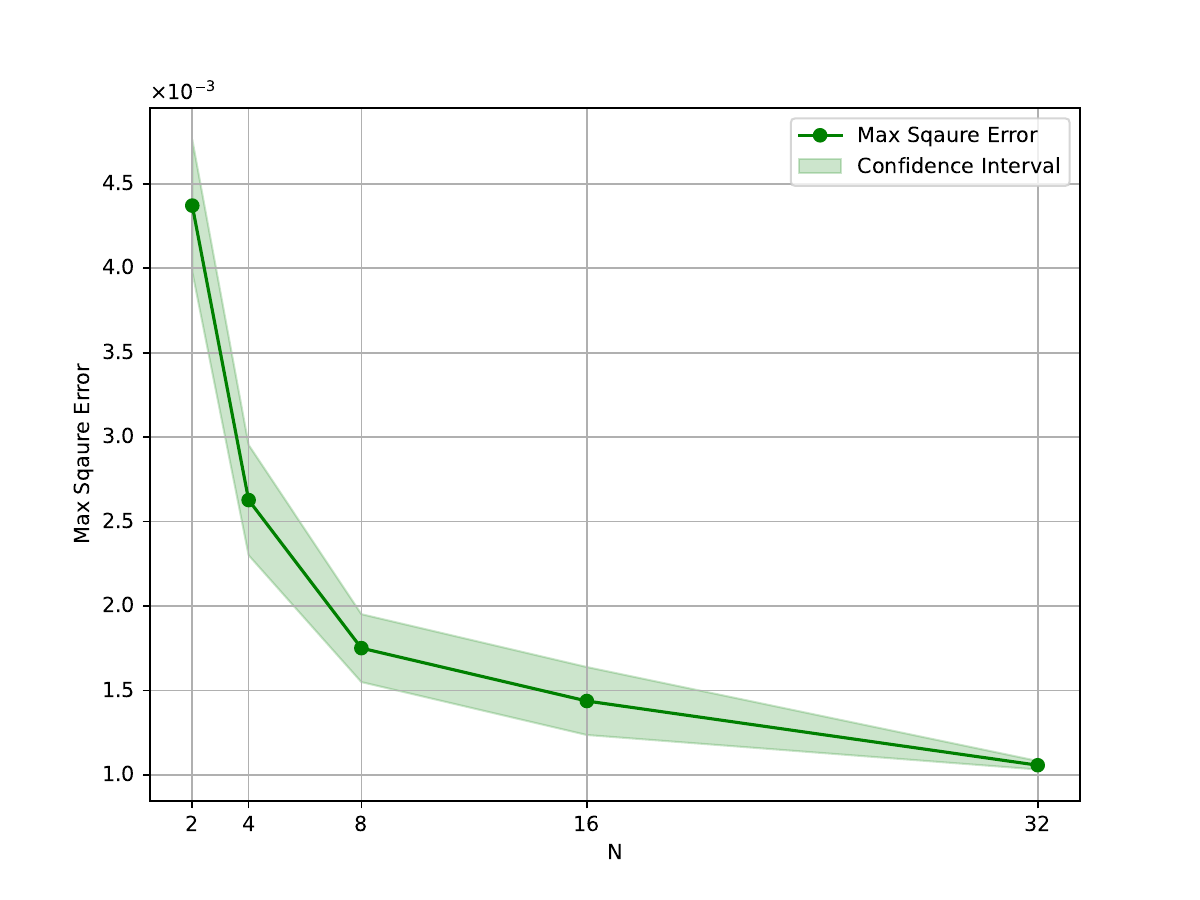}
            \caption{Max square error}
        \end{subfigure}   
    \end{tabular}
    \caption{Numerical results with different time step-size $\Delta t=1/N$. The horizontal coordinate is the number of time step $N$.}
    \label{fig:different_N}
\end{figure}
The three graphs in Fig. \ref{fig:different_N} show the variation of three metrics with different numbers of time step $N$. The graph (A) shows that the loss function decreases significantly as the number of time step increases. The graph (B) indicates that the mean relative error decreases as the number of time step increases. When the number of time step is $2$, the average relative error is approximately $0.04$. When the number of time step increases to $32$, the average relative error decreases to nearly $0.01$. The graph (C) shows that the maximum squared error also decreases as the number of time step increases. Overall, as the number of time step increases (i.e., the time step-size becomes smaller), all three metrics decrease, suggesting that the model performs better and has higher prediction accuracy with smaller time steps.

\begin{table}[h]
    \centering
    \caption{Relationship between max square error and time step-size, and the convergence order.}
    \label{tab:different_N_order}
    \begin{tabular}{cccc}
        \toprule
        Number of step $N$ & Time step-size & Max square error & Convergence order \\
        \midrule
        2 & 0.50000 & 4.371e-03 & - \\
        4 & 0.25000 & 2.628e-03 & 0.73 \\
        8 & 0.12500 & 1.752e-03 & 0.59 \\
        16 & 0.06250 & 1.438e-03 & 0.28 \\
        32 & 0.03125 & 1.058e-03 & 0.44 \\
        \bottomrule
        \end{tabular}
\end{table}
Table \ref{tab:different_N_order} shows the maximum square error and its corresponding convergence order for different time step-sizes. As the time step-size decreases, the maximum square error decreases as well, indicating an improvement in the accuracy of the numerical method. Specifically, as the time step reduces from $0.5$ to $0.03125$, the maximum square error decreases from 4.371e-03 to 1.058e-03. This is consistent with our theoretical finds.

At the same time, by calculating the order, we observe that the convergence order in error slows down, from $0.60$ to $0.44$. This suggests that while the accuracy improves, the convergence rate of the error becomes more complex at smaller time steps. This may be due to the optimizer not yet being able to find the optimal parameters of the loss function and the neural network approximation error being dominant when the time discretization error is small.

\section{Concluding remarks}

In this paper, we proposed the forward-backward stochastic jump neural network (FBSJNN) method, a novel deep learning-based method for solving PIDEs and FBSDEJs. In this method, only the solution itself is needed to approximate by neural network. This enables the method to reduce the total number of parameters in FBSJNN, which enhances optimization efficiency. Leveraging principles from stochastic calculus and the universal approximation theorem, we also showed that neural networks can effectively approximate the solutions to complex stochastic problems theoretically. We provided the consistency error estimates of the FBSJNN, highlighting the distinctions between method error and universal error in the context of neural network approximations. Extensive numerical experiments validate the method’s rapid convergence and high accuracy across diverse scenarios, underscoring its superiority over existing techniques.
Exploring the extension of this methodology to more general forms of stochastic differential equations and PIDEs and the development of adaptive algorithms for optimizing neural network architectures will be our future work.

\bibliographystyle{amsplain}
\bibliography{FBSJNN.bib}

\end{document}